%% file: Riemann-15.tex
\documentclass[12pt,reqno]{amsart}
\usepackage{amsmath,a4wide}
\usepackage{stmaryrd,mathrsfs,bm,amsthm,mathtools,yfonts,amssymb}
\usepackage[pdftex]{graphics,color}

\numberwithin{equation}{section}

\begingroup
\newtheorem{theorem}{Theorem}[section]
\newtheorem{lemma}[theorem]{Lemma}
\newtheorem{proposition}[theorem]{Proposition}
\newtheorem{corol}[theorem]{Corollary}
\endgroup

\theoremstyle{definition}
\newtheorem{definition}[theorem]{Definition}
\newtheorem{remark}[theorem]{Remark}

\allowdisplaybreaks[3]
\renewcommand{\div}{\mathop{\rm div}\nolimits}
\renewcommand{\det}{\mathop{\rm det}\nolimits}
\newcommand{\tr}{\mathop{\rm tr}\nolimits}

\newcommand{\Sym}{\mathcal{S}}
\newcommand{\R}{\mathbb{R}}
\newcommand{\N}{\mathbb{N}}

\newcommand{\dt}{\partial_t}
\newcommand{\dx}{\partial_{x_2}}
\providecommand{\abs}[1]{\left\lvert#1\right\rvert}

\newcommand{\id}{{\rm Id}}
\newcommand\eps{{\varepsilon}}

\author{Elisabetta Chiodaroli, Camillo De Lellis and Ond\v{r}ej Kreml}
\title[Global ill-posedness for compressible Euler]{Global ill-posedness of the
isentropic system of gas dynamics}

\begin{document}
 
\begin{abstract}
We consider the isentropic compressible Euler system in $2$ space dimensions
with pressure law $p (\rho) = \rho^2$ and we show the existence of classical
Riemann data, i.e. pure jump discontinuities across a line, for which there
are infinitely many admissible bounded weak solutions (bounded away from the void).
We also show that some of these Riemann data are generated by a $1$-dimensional
compression wave: our theorem leads therefore to Lipschitz initial data for which
there are infinitely many global bounded admissible weak solutions.
\end{abstract}

\maketitle

\section{Introduction}

Consider the isentropic compressible Euler
equations of gas dynamics in two space dimensions.
This system consists of $3$ scalar equations, which state the conservation of mass and linear momentum.
The unknowns are the density $\rho$ and the velocity $v$. The resulting Cauchy problem
takes the form:
\begin{equation}\label{eq:Euler system}
\left\{\begin{array}{l}
\partial_t \rho + {\rm div}_x (\rho v) \;=\; 0\\
\partial_t (\rho v) + {\rm div}_x \left(\rho v\otimes v \right) + \nabla_x [ p(\rho)]\;=\; 0\\
\rho (\cdot,0)\;=\; \rho^0\\
v (\cdot, 0)\;=\; v^0 \, .
\end{array}\right.
\end{equation}
The pressure $p$ is a function of $\rho$ determined from the constitutive thermodynamic
relations of the gas under consideration and it is assumed to satisfy $p'>0$ (this
hypothesis guarantees also the hyperbolicity of the system on the regions where $\rho$ is
positive).
A common choice is the polytropic pressure law
$p(\rho)= \kappa \rho^\gamma$
with constants $\kappa>0$ and $\gamma>1$. The classical kinetic theory of gases predicts 
exponents $\gamma = 1 + \frac{2}{d}$, where $d$ is the degree of freedom of the molecule
of the gas. Here we will be concerned mostly with the particular choice 
$p(\rho)= \rho^2$. However several of our technical statements hold under the general assumption
$p'>0$ and the specific choice $p (\rho) = \rho^2$ is relevant only to some portions of
our proofs.

It is well-known that, even starting from extremely regular initial data, the system 
\eqref{eq:Euler system}
develops singularities in finite time. In the mathematical literature a lot of effort has
been devoted to understanding how solutions can be continued after the appearance of the first
singularity, leading to a
quite mature theory in one space dimension (we refer the reader to the monographs
\cite{br},\cite{da} and \cite{se}). In this paper we show that, in more than one
space dimension, 
the most popular concept of an admissible solution fails to yield uniqueness even under
very strong assumptions on the initial data. In particular we consider bounded
weak solutions of \eqref{eq:Euler system}, satisfying \eqref{eq:Euler system} in the usual
distributional sense (we refer to Definition \ref{d:weak} for the
precise formulation), and we call them admissible if they satisfy the following additional
inequality in the sense of distibutions (usually called {\em entropy inequality}, although
for the specific system \eqref{eq:Euler system} it is rather a weak form of energy balance):
\begin{equation} \label{eq:energy inequality}
\dt \left(\rho \varepsilon(\rho)+\rho
\frac{\abs{v}^2}{2}\right)+\div_x
\left[\left(\rho\varepsilon(\rho)+\rho
\frac{\abs{v}^2}{2}+p(\rho)\right) v \right]
\;\leq\; 0
\end{equation}
where the internal energy $\varepsilon: \R^+\rightarrow \R$ is given
through the law $p(r)=r^2 \varepsilon'(r)$. Indeed, admissible solutions are required
to satisfy a slightly stronger condition, i.e. a form of \eqref{eq:energy inequality}
which involves also the initial data, see Definition \ref{d:admissible}.
For all solutions considered in this paper, $\rho$ will always be
bounded away from $0$, i.e. $\rho\geq c_0$ for some positive constant $c_0$.

We denote the space variable as $x=(x_1, x_2)\in \R^2$ and consider the special
initial data
\begin{equation}\label{eq:R_data}
(\rho^0 (x), v^0 (x)) := \left\{
\begin{array}{ll}
(\rho_-, v_-) \quad & \mbox{if $x_2<0$}\\ \\
(\rho_+, v_+) & \mbox{if $x_2>0$,} 
\end{array}\right. 
\end{equation}
where $\rho_\pm, v_\pm$ are constants. 
It is well-known that for some special choices of these constants there are
solutions of \eqref{eq:Euler system} which are {\em rarefaction waves}, i.e.
self-similar solutions depending only on $t$ and $x_2$ which are locally Lipschitz
for positive $t$ and constant on lines emanating from the origin (see \cite[Section 7.6]{da}
for the precise definition). Reversing their order (i.e. exchanging $+$ and $-$) the
very same constants allow for a {\em compression wave} solution, i.e. a solution on
$\R^2\times ]-\infty, 0[$ which is locally Lipschitz and converges, for $t\uparrow 0$,
to the jump discontinuity of \eqref{eq:R_data}. When this is the case we will then say that
the data \eqref{eq:R_data} are {\em generated by a classical compression wave}.

We are now ready to state the main theorem of this paper

\begin{theorem}\label{t:main}
Assume $p (\rho) = \rho^2$. Then there are data as in \eqref{eq:R_data} for
which there are infinitely many bounded admissible solutions $(\rho, v)$ of
\eqref{eq:Euler system} on $\R^2\times [0, \infty[$ 
with $\inf \rho >0$. Moreover, these data are generated by classical compression waves.  
\end{theorem}

It follows from the usual treatment of the $1$-dimensional Riemann problem that for
the data of Theorem \ref{t:main} uniqueness holds if
the admissible solutions are also required to be self-similar, i.e.
of the form $(\rho, v) (x,t) = \left(r \left(\frac{x_2}{t}\right), 
w \left(\frac{x_2}{t}\right)\right)$ and to have locally bounded variation (see Proposition
\ref{p:Riemann_classico}). Note that such solutions must be discontinuous, because the
data of Theorem \ref{t:main} are generated by compression waves. We in fact conjecture that this is
the case for {\em any} initial data \eqref{eq:R_data} allowing the nonuniqueness property
of Theorem \ref{t:main}: 
however this fact does not seem to follow from the usual
weak-strong uniqueness (as for instance in \cite[Theorem 5.3.1]{da}) because the Lipschitz
constant of the classical solution blows up as $t\downarrow 0$. Related results
in one space dimension are contained in the work of DiPerna \cite{DP} and in the works of
Chen and Frid \cite{CF1}, \cite{CF2}.

As an obvious corollary of Theorem \ref{t:main} we arrive at the following statement. 

\begin{corol}\label{c:compression}
There are Lipschitz initial data $(\rho^0, v^0)$ for which there are infinitely
many bounded admissible solutions $(\rho, v)$ of \eqref{eq:Euler system} on
$\R^2\times [0, \infty[$ with
$\inf \rho >0$. These
solutions are all locally Lipschitz on a finite interval on which they
all coincide with the unique classical solution.   
\end{corol}

We note in passing that, although the last statement of the corollary can be directly proved
following the details of our construction, it is also a consequence of the admissibility
condition, the Lipschitz regularity of the compression wave (before the singular time
is reached) and the well-known weak-strong uniqueness of \cite[Theorem 5.3.1]{da}.

\subsection{$h$-principle and the Euler equations} The proof of Theorem \ref{t:main} 
relies heavily on the works of the second author and L\'aszl\'o Sz\'ekelyhidi, who 
in the paper \cite{dls1} introduced methods from the theory of differential
inclusions to explain the existence of compactly supported nontrivial weak solutions of the
{\em incompressible} Euler equations (discovered in the pioneering work
of Scheffer \cite{sc}; see also \cite{sh}). It was already observed by the same pair of
authors that these methods could be applied to the compressible Euler equations and lead
to the ill-posedness of bounded admissible solutions, see \cite{dls2}. However, the data of
\cite{dls2} were extremely irregular and raised the question whether the ill-posedness was
due to the irregularity of the data, rather than to the irregularity of the solution. 

A preliminary answer
was provided in the work \cite{ch} where the first author showed that data with very regular densities
but irregular velocities still allow for nonuniqueness of admissible solutions. The
present paper gives a complete answer, since we show that even for some smooth initial data 
nonuniqueness of bounded admissible solutions arises after the first blow-up time. 
It remains however an open question how irregular
such solutions have to be in order to display the pathological behaviour of Theorem \ref{t:main}. 
One could speculate that, in analogy to what has been shown recently for the incompressible
Euler equations, even a ``piecewise H\"older regularity'' might not be enough; see
\cite{dls4}, \cite{dls5}, \cite{isett}, \cite{bdls} and in particular \cite{dan}.

This paper draws also heavily from the work \cite{sz} where Sz\'ekelyhidi 
coupled the methods introduced in \cite{dls1}-\cite{dls2} with a clever construction
to produce rather surprising irregular solutions of the incompressible Euler equations 
with vortex-sheet initial data. This work of Sz\'ekelyhidi was in turn motivated by the so-called Muskat problem
(see \cite{cfg}, \cite{sz2} and \cite{shvidkoy}; we moreover refer to \cite{dls3} for a rather detailed
survey). Indeed the basic idea of looking for piecewise constant subsolutions as defined in Section \ref{s:subsolution}
stems out of several conversations with Sz\'ekelyhidi and have been inspired by a remark of Shnirelman upon
the proof of \cite{sz}.

\subsection{Acknowledgements} The research of Camillo De Lellis has been supported by the SNF Grant 129812,
whereas Ond\v{r}ej Kreml's research has been financed by the SCIEX Project 11.152. The authors are also
very thankful to L\'aszl\'o Sz\'ekelyhidi for several enlightning conversations.

\section{Ideas of the proof and plan of the paper}

\subsection{Subsolutions}\label{ss:subsolutions} 
Especially relevant for us is the appropriate notion of {\em subsolution},
which allows to use the methods of \cite{dls1}-\cite{dls2} to solve the equations {\em and}
impose a certain specific initial data. We give here a brief description of the concept of
subsolution relevant to us and refer to \cite{dls3} for the motivation behind it and its
links to existing literature in physics and mathematics.

Consider first some data as in \eqref{eq:R_data}. We then partition the upper half space
$\{t>0\}$ in regions contained between half-planes meeting all at the line $\{t=x_2=0\}$, see
Definition \ref{d:fan} and cf. Figure \ref{fig:fan}. We then define the density function 
$\rho = \overline{\rho}$ to be constant in each region: this density function will indeed
give the final $\rho$ for all the solutions we construct and it is therefore required to
take the constant values $\rho_\pm$ in the outermost regions $P_\pm$.

\begin{figure}[htbp]
\begin{center}
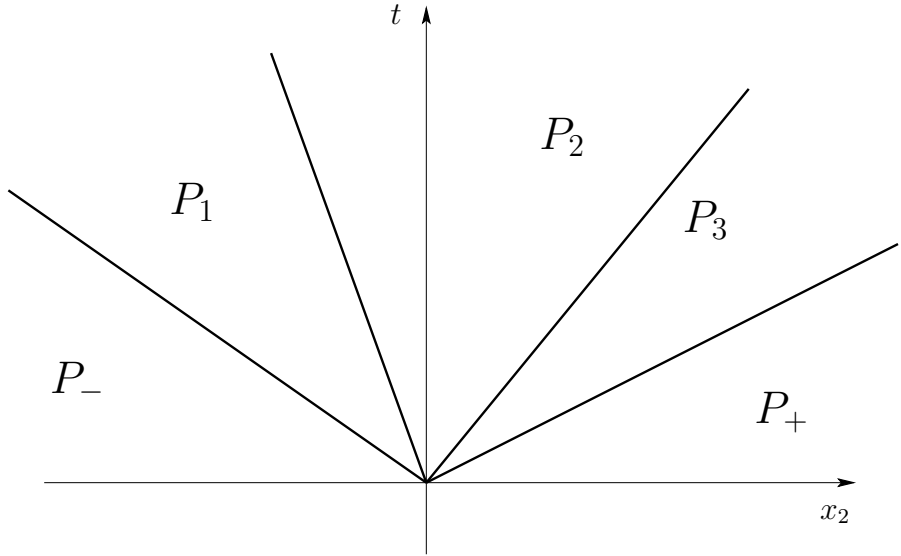
\caption{A ``fan partition'' in five regions.}
\label{fig:fan}
\end{center}
\end{figure}

We then solve the compressible Euler equations \eqref{eq:Euler system} in each region
$P_1, \ldots, P_N$ using the methods of \cite{dls2}. Indeed observe that in each such region 
the density is constant and thus it suffices to
to construct solutions of the {\em incompressible} Euler equations
with {\em constant pressure}. Employing the methods of \cite{dls2} we can also impose
that the modulus of the velocity is constant (in each region): its square 
will be denoted by $C_i$. 
In \cite{dls2} such solutions are constructed adding oscillations to an appropriate
{\em subsolution}, which consists of a pair $\overline{v}, \overline{u}$ of smooth functions,
the first taking values in $\R^2$ and the second taking values in the space of symmetric, trace-free
$2\times 2$ matrices. These functions satisfy the linear system of PDEs 
\[
\left\{\begin{array}{l}
\partial_t \overline{v} + {\rm div}_x \overline{u} = 0\\ \\
{\rm div}_x \overline{v} =0\, .
\end{array}\right.
\]
and a suitable relaxation of the nonlinear constraints $u = v\otimes v - \frac{|v|^2}{2} {\rm Id}$. 

In our particular case we will choose our subsolutions to be {\em constant} on each 
region $P_i$: the corresponding values will be denoted by $(\rho_i, v_i, u_i)$ and the
corresponding globally defined (piecewise constant) functions $(\overline{\rho}, \overline{v}, \overline{u})$
will be called {\em fan subsolutions}
of the compressible Euler equations. 
We then wish to choose our subsolution so that, after solving
\eqref{eq:Euler system} in each region $P_i$ with the methods of \cite{dls2}, the resulting globally 
defined $(\rho, v)$ are admissible {\em global} solutions of \eqref{eq:Euler system}. This leads to a suitable
system of PDEs for the piecewise constant functions $(\overline{\rho}, \overline{v}, \overline{u})$
which are
summarized in the Definitions \ref{d:subs} and \ref{d:admiss}. In Section \ref{s:subsolution}
we then briefly recall the notions of the papers \cite{dls1}-\cite{dls2} and in Section
\ref{s:CI} we describe how to
suitably modify the arguments there to reduce the proof of Theorem \ref{t:main} to the
existence of the ``fan subsolutions'' of Definitions \ref{d:subs} and \ref{d:admiss}:
the precise statement of this reduction is given in Proposition \ref{p:subs}.

\subsection{The algebraic system} In Section
\ref{s:algebra}, by making some specific choices, the existence of such subsolution is 
reduced to finding an array of real numbers satisfying some algebraic identities and inequalities,
see Proposition \ref{p:algebra}. Indeed, since the functions $(\overline{\rho}, \overline{v}, \overline{u})$
assume constant values in each region of the fan decomposition, these conditions are nothing but suitable
``Rankine-Hugoniot type'' identitites and inequalities. Although at this stage all computations can be
carried in general, we restrict our attention to a fan decomposition which consists of only three regions.
Therefore, the resulting solutions provided by Proposition \ref{p:subs} (and therefore
also those of Theorem \ref{t:main}) will take the constant values
$(\rho_\pm, v_\pm)$ outside a ``wedge'' of the form $P_1 = \{\nu_- t < x_2< \nu_+ t\}$: inside this wedge
the solutions will instead behave in a very chaotic way. 

Thus far, all the statements can be carried out for a general
pressure law $p$. In the case $p (\rho) = \rho^2$ we also compute explicitely the well-known conditions
that must be imposed on the velocities $v_\pm$ and $\rho_\pm$ so that the corresponding data
\eqref{eq:R_data} are generated by a compression wave: this gives then an additional constraint.
Observe that for such data the ``classical solution'' will be a simple shock wave  traveling at a certain
speed, whereas the nonstandard solutions of Theorem \ref{t:main} ``open up'' the singularity and fill
the corresponding region $P_1$ with many oscillations.

Coming back to the algebraic constrains of Proposition \ref{p:algebra}, although 
there seems to be a certain abundance of solutions to this set of
identities and inequalities, currently we do not have an efficient and general 
method for finding them. We propose two
possible ways in the Sections \ref{s:rho^2_1} and \ref{s:p_strana}. That
of Section \ref{s:rho^2_1} is the most effective and produces the initial data of Theorem \ref{t:main}
which are generated by a compression wave. That of Section \ref{s:p_strana} is an alternative
strategy, where, instead of making a precise choice
of the pressure law $p$, we exploit it as an extra degree of freedom: 
as a result this method gives data as in Theorem \ref{t:main} but with a different pressure law,
which is essentially a suitable smoothing of the step-function. We also do not know whether any of
these data are generated by compression waves. 

\subsection{Classical Riemann problem} Finally in Section
\ref{s:riemann} we show that the self-similar solutions to \eqref{eq:Euler system}-\eqref{eq:R_data}
are unique: this follows from classical considerations
but since we have not been able to find a precise reference, we include the argument for completeness.

\section{Subsolutions}\label{s:subsolution}

\subsection{Weak and admissible solutions of \eqref{eq:Euler system}}
We recall here the usual definitions of weak and admissible
solutions to \eqref{eq:Euler system}.

\begin{definition}\label{d:weak}
By a \textit{weak solution} of \eqref{eq:Euler system} on $\R^2\times[0,\infty[$ we
mean a pair $(\rho, v)\in L^\infty(\R^2\times [0,\infty[)$ such that the following identities 
hold for every test functions $\psi\in C_c^{\infty}(\R^2\times [0, \infty[)$,
$\phi\in C_c^{\infty}(\R^2\times [0, \infty[)$:
\begin{equation} \label{eq:weak1}
\int_0^\infty \int_{\R^2} \left[\rho\dt \psi+ \rho v \cdot \nabla_x \psi\right] dx dt+\int_{\R^2} \rho^0(x)\psi(x,0) dx \;=\; 0
\end{equation}
\begin{align} \label{eq:weak2}
&\int_0^\infty \int_{\R^2} \left[ \rho v \cdot \dt \phi+ \rho v \otimes v : D_x \phi +p(\rho) \div_x \phi \right]
+\int_{\R^2} \rho^0(x) v^0(x)\cdot\phi(x,0) dx\;=\; 0.
\end{align}
\end{definition}

\begin{definition}\label{d:admissible}
A bounded weak solution $(\rho, v)$ of \eqref{eq:Euler system} is admissible if
it satisfies the following inequality for every nonnegative 
test function $\varphi\in C_c^{\infty}(\R^2\times [0,\infty[)$:
\begin{align} \label{eq:admissibility condition}
 &\int_0^\infty\int_{\R^2} \left[\left(\rho\varepsilon(\rho)+\rho \frac{\abs{v}^2}{2}\right)\dt \varphi+\left(\rho\varepsilon(\rho)+\rho
\frac{\abs{v}^2}{2}+p(\rho)\right) v \cdot \nabla_x \varphi \right]\notag \\
&+\int_{\R^2} \left(\rho^0 (x) \varepsilon(\rho^0 (x))+\rho^0 (x)\frac{\abs{v^0 (x)}^2}{2}\right) 
\varphi(x,0)\, dx \;\geq\; 0\, .
\end{align}
\end{definition}

\subsection{Subsolutions} To begin with, we state more precisely the definition of subsolution in our context. 
Here $\Sym_0^{2\times2}$ denotes the set of symmetric traceless $2\times2$ matrices and 
$\id$ is the identity matrix. We first introduce a notion of good partition for the upper half-space
$\R^2\times ]0, \infty[$.

\begin{definition}[Fan partition]\label{d:fan}
A {\em fan partition} of $\R^2\times ]0, \infty[$ consists of finitely many open sets $P_-, P_1, \ldots ,P_N, P_+$
of the following form 
\begin{align}
 P_- &= \{(x,t): t>0 \quad \mbox{and} \quad x_2 < \nu_- t\}\\
 P_+ &= \{(x,t): t>0 \quad \mbox{and} \quad x_2 > \nu_+ t\}\\
 P_i &= \{(x,t): t>0 \quad \mbox{and} \quad \nu_{i-1} t < x_2 < \nu_i t\}
\end{align}
where $\nu_- = \nu_0 < \nu_1 < \ldots < \nu_N = \nu_+$ is an arbitrary collection of real numbers.
\end{definition}

The next two definitions are then motivated by the discussion of Section \ref{ss:subsolutions}. However
at the present stage it is not completely clear why the relevant partial differential
equations (and inequalities!) for the piecewise constant solutions are given by \eqref{eq:continuity},
\eqref{eq:momentum} and \eqref{eq:admissible subsolution}: their role will become
transparent in the next subsection when we prove Proposition \ref{p:subs}. 

\begin{definition}[\textit{Fan Compressible} subsolutions] \label{d:subs}
A \textit{fan subsolution} to the compressible Euler equations \eqref{eq:Euler system} with
initial data \eqref{eq:R_data} is a triple 
$(\overline{\rho}, \overline{v}, \overline{u}): \R^2\times 
]0,\infty[ \rightarrow (\R^+, \R^2, \Sym_0^{2\times2})$ of piecewise constant functions satisfying
the following requirements.
\begin{itemize}
\item[(i)] There is a fan partition $P_-, P_1, \ldots, P_N, P_+$ of $\R^2\times ]0, \infty[$ such that
\[
(\overline{\rho}, \overline{v}, \overline{u})= \sum_{i=1}^N 
(\rho_i, v_i, u_i) \bm{1}_{P_i}
+ (\rho_-, v_-, u_-) \bm{1}_{P_-}
+ (\rho_+, v_+, u_+) \bm{1}_{P_+}
\]
where $\rho_i, v_i, u_i$ are constants with $\rho_i >0$ and $u_\pm =
v_\pm\otimes v_\pm - \textstyle{\frac{1}{2}} |v_\pm|^2 \id$;
\item[(ii)] For every $i\in \{1, \ldots, N\}$ there exists a positive constant $C_i$ such that
\begin{equation} \label{eq:subsolution 2}
v_i\otimes v_i- u_i < \frac{C_i}{2} \id\, .
\end{equation}
\item[(iii)] The triple $(\overline{\rho}, \overline{v}, \overline{u})$ solves the following system in the
sense of distributions:
\begin{align}
&\partial_t \overline{\rho} + {\rm div}_x (\overline{\rho} \, \overline{v}) \;=\; 0\label{eq:continuity}\\
&\partial_t (\overline{\rho} \, \overline{v})+{\rm div}_x \left(\overline{\rho} \, \overline{u} 
\right) + \nabla_x \left( p(\overline{\rho})+\frac{1}{2} \left(\sum_i C_i \rho_i
\bm{1}_{P_i} + \overline{\rho} |\overline{v}|^2 \bm{1}_{P_+\cup P_-}\right)\right)= 0\label{eq:momentum}
\end{align}
\end{itemize}
\end{definition}

\begin{definition}[Admissible fan subsolutions]\label{d:admiss}
 A fan subsolution $(\overline{\rho}, \overline{v}, \overline{u})$ is said to be {\em admissible}
if it satisfies the following inequality in the sense of distributions
\begin{align} 
&\dt \left(\overline{\rho} \varepsilon(\overline{\rho})\right)+\div_x
\left[\left(\overline{\rho}\varepsilon(\overline{\rho})+p(\overline{\rho})\right) \overline{v}\right]
 + \dt \left( \overline{\rho} \frac{|\overline{v}|^2}{2} \bm{1}_{P_+\cup P_-} \right)
+ \div_x \left(\overline{\rho} \frac{|\overline{v}|^2}{2} \overline{v} \bm{1}_{P_+\cup P_-}\right)\nonumber\\
&\qquad\qquad+ \sum_{i=1}^N \left[\dt\left(\rho_i \, \frac{C_i}{2} \, \bm{1}_{P_i}\right) 
+ \div_x\left(\rho_i \, \overline{v} \, \frac{C_i}{2}  \, \bm{1}_{P_i}\right)\right]
\;\leq\; 0\, .\label{eq:admissible subsolution}
\end{align}
\end{definition}

It is possible to generalize these notions in several directions, e.g. allowing partitions
with more general open sets and functions $v_i, u_i$ and $\rho_i$ which vary (for instance continuously) 
in each element of the partition. It is not difficult to extend the conclusions of the next subsection
to such settings. However we have chosen to keep the definitions to the minimum needed for our proof
of Theorem \ref{t:main}.

\subsection{Reduction to admissible fan subsolutions} Using the techniques introduced in \cite{dls1}-\cite{dls2}
we then reduce Theorem \ref{t:main} to finding an admissible fan subsolution. The precise statement is given in
the following proposition.

\begin{proposition}\label{p:subs}
Let $p$ be any $C^1$ function and $(\rho_\pm, v_\pm)$ be such that there exists at least one
admissible fan subsolution $(\overline{\rho}, \overline{v}, \overline{u})$ of \eqref{eq:Euler system}
with initial data \eqref{eq:R_data}. Then there are infinitely 
many bounded admissible solutions $(\rho, v)$ to \eqref{eq:Euler system}-\eqref{eq:R_data} such that 
$\rho=\overline{\rho}$.
\end{proposition}

The core of the proof is in fact a corresponding statement for subsolutions of the {\em incompressible
Euler equations} which is essentially contained in the proofs of \cite{dls1}-\cite{dls2}. However, since our
assumptions and conlusions are slightly different, we state them in the next lemma.

\begin{lemma}\label{l:ci}
Let $(\tilde{v}, \tilde{u})\in \R^2\times \Sym_0^{2\times 2}$ and $C>0$ be such that $\tilde{v}\otimes \tilde{v}
- \tilde{u} < \frac{C}{2} \id$. For any open set $\Omega\subset \R^2\times \R$ there are infinitely many maps
$(\underline{v}, \underline{u}) \in L^\infty (\R^2\times \R , \R^2\times \Sym_0^{2\times 2})$ with the following property
\begin{itemize}
\item[(i)] $\underline{v}$ and $\underline{u}$ vanish identically outside $\Omega$;
\item[(ii)] $\div_x \underline{v} = 0$ and $\partial_t \underline{v} + \div_x \underline{u} = 0$;
\item[(iii)] $ (\tilde{v} + \underline{v})\otimes (\tilde{v} + \underline{v}) - (\tilde{u} + \underline{u}) = \frac{C}{2} \id$
a.e. on $\Omega$.
\end{itemize}
\end{lemma}

The proof is a minor variant of the ones given in \cite{dls1}-\cite{dls2} but since none of the
statements present in the literature matches exactly the one of Lemma \ref{l:ci} we give some of the
details in the next Section, referring to precise lemmas in the papers \cite{dls1}-\cite{dls2}. For the moment
we show how Proposition \ref{p:subs} derives from Lemma \ref{l:ci}.

\begin{proof}[Proof of Proposition \ref{p:subs}] We apply Lemma \ref{l:ci} in each region $\Omega = P_i$ and
we call $(\underline{v}_i, \underline{u}_i)$ any pair of maps given by such Lemma. Hence we set
\begin{align}
& v : = \overline{v} + \sum_{i=1}^N \underline{v}_i\\
& u := \overline{u} + \sum_{i=1}^N \underline{u}_i
\end{align}
whereas $\rho = \overline{\rho}$ (as claimed in the statement of the Proposition!). We next show that 
the pair $(\rho, v)$ is an admissible weak solution of \eqref{eq:Euler system}-\eqref{eq:R_data}. First observe that
$\div_x (\rho_i \underline{v}_i) =0$ since $\rho_i$ is a constant. But since $\underline{v}_i$ is supported in $P_i$
and $\rho = \overline{\rho} \equiv \rho_i$ on $P_i$, we then conclude $\div_x (\overline{\rho} \underline{v}_i) = 0$. Thus 
we have 
\begin{align}
&\partial_t \rho + \div_x (\rho v) = \partial_t \overline{\rho} + \div_x \left(\overline{\rho} \overline{v} +
\sum_i \overline{\rho} \underline{v}_i\right) \nonumber\\
=&\partial_t \overline{\rho} + \div_x (\overline{\rho} \overline{v}) + \sum_i \div_x (\overline{\rho} \underline{v}_i) 
= \partial_t \overline{\rho} + \div_x (\overline{\rho} \overline{v}) = 0\, 
\end{align}
in the sense of distributions.
Moreover, observe that
\[
v\otimes v =
\left\{
\begin{array}{ll}
v_+\otimes v_+ &\mbox{on $P_+$}\\
v_-\otimes v_- &\mbox{on $P_-$}\\
(v_i + \underline{v}_i)\otimes (v_i+\underline{v}_i) = u_i + \underline{u}_i + \frac{C_i}{2}\id \quad &\mbox{on $P_i$} 
\end{array}
\right.
\]
and
\[
\overline{u} =
\left\{
\begin{array}{ll}
v_+\otimes v_+ - \frac{1}{2} |v_+|^2 \id \quad&\mbox{on $P_+$}\\ \\
v_-\otimes v_- - \frac{1}{2} |v_-|^2 \id \quad&\mbox{on $P_-$}\\ \\
u_i &\mbox{on $P_i$}\, . 
\end{array}
\right.
\]
Moreover, on each region $P_i$ we have
\[
\rho v \otimes v =  \rho \left(v\otimes v - \frac{|v|^2}{2} {\rm Id}\right) + \frac{\rho C_i}{2}  {\rm Id} 
= \overline{\rho} \overline{u} + \rho_i \underline{u}_i + \frac{C_i \rho_i}{2}  {\rm Id}\, .
\]
Hence, we can write
\begin{align}
& \partial_t (\rho v) + \div_x (\rho v\otimes v) + \nabla_x
 [p (\rho)] = \partial_t \left(\overline{\rho}\overline{v} +
\sum_i \rho_i \underline{v}_i\right) + \div_x \left(\overline{\rho}\overline{u} + \sum_i \rho_i \underline{u}_i\right)\nonumber\\
&\qquad + \nabla_x \left(p (\overline{\rho}) + \frac{1}{2} \sum_i C_i\rho_i \bm{1}_{P_i} + \frac{1}{2} |v_-|^2 \rho_- \bm{1}_{P_-}
+ \frac{1}{2} |v_+|^2 \rho_+ \bm{1}_{P_+}\right)\nonumber\\
= & \partial_t \left(\overline{\rho}\overline{v}\right) +  \div_x \left(\overline{\rho}\overline{u}\right) +
\nabla_x \left(p (\overline{\rho}) + \frac{1}{2} \sum_i C_i\rho_i \bm{1}_{P_i} + \frac{1}{2} |v_-|^2 \rho_- \bm{1}_{P_-}
+ \frac{1}{2} |v_+|^2 \rho_+ \bm{1}_{P_+}\right)\nonumber\\
&\qquad + \sum_i \rho_i \underbrace{\partial_t \underline{v}_i + \div_x \underline{u}_i}_{= 0}\, .
\end{align}
Therefore, by Definition \ref{d:subs} we conclude $\partial_t (\rho v) + \div_x (\rho v\otimes v) + \nabla_x [p (\rho)] = 0$.

Next, we compute 
\begin{align}
& \partial_t \left(\rho \eps (\rho) + \frac{|v|^2}{2} \rho\right) + \div_x \left(\left(\rho \eps (\rho) + \frac{|v|^2}{2} \rho
+ p (\rho)\right)v\right)\nonumber\\
= & \partial_t \left(\overline{\rho}\eps (\overline{\rho}) + \sum_i \frac{1}{2} C_i \rho_i \bm{1}_{P_i} + \frac{|v_-|^2}{2} \rho_-
\bm{1}_{P_-} + \frac{|v_+|^2}{2} \rho_+ \bm{1}_{P_+}\right)\nonumber\\
&\quad + \div_x \left[ \left(\overline{\rho}\eps (\overline{\rho}) + p (\overline{\rho}) + \sum_i \frac{1}{2} C_i \rho_i \bm{1}_{P_i} + \frac{|v_-|^2}{2} \rho_-
\bm{1}_{P_-} + \frac{|v_+|^2}{2} \rho_+ \bm{1}_{P_+}\right)\left( \overline{v} + \sum_i \underline{v}_i\right)\right]\nonumber
\end{align}
Using the condition \eqref{eq:admissible subsolution} we therefore conclude 
\begin{align}
& \partial_t \left(\rho \eps (\rho) + \frac{|v|^2}{2} \rho\right) + \div_x \left(\left(\rho \eps (\rho) + \frac{|v|^2}{2} \rho
+ p (\rho)\right)v\right)\nonumber\\
\leq & \sum_i \div_x \Big[ \underline{v}_i \Big(\underbrace{\overline{\rho}\eps (\overline{\rho}) + p (\overline{\rho}) + \sum_i \frac{1}{2} C_i \rho_i \bm{1}_{P_i} + \frac{|v_-|^2}{2} \rho_-
\bm{1}_{P_-} + \frac{|v_+|^2}{2} \rho_+ \bm{1}_{P_+}}_{=: \varrho}\Big)\Big]
\end{align}
in the sense of distributions. Observe however that the function $\varrho$ is constant on each $P_i$, 
on which $\underline{v}_i$ is 
supported. Thus
\begin{align}
& \partial_t \left(\rho \eps (\rho) + \frac{|v|^2}{2} \rho\right) + \div_x \left(\left(\rho \eps (\rho) + \frac{|v|^2}{2} \rho
+ p (\rho)\right)v\right) \leq \sum_i \varrho \div_x \underline{v}_i = 0\, .
\end{align}
So far we have shown that \eqref{eq:weak1}, \eqref{eq:weak2} and \eqref{eq:admissibility condition}
hold whenever the corresponding test functions are supported in $\R^2\times ]0, \infty[$. However observe that,
since as $\tau\downarrow 0$ the Lebesgue measure of $P_i \cap \{t = \tau\}$ converges to $0$, 
the maps $\rho (\cdot , \tau)$ and $v(\cdot , \tau)$ converge to the maps $\rho^0$ and $v^0$ of \eqref{eq:R_data}
strongly in $L^1_{loc}$. This easily implies \eqref{eq:weak1}, \eqref{eq:weak2} and \eqref{eq:admissibility condition}
in their full generality. For instance, assume $\psi\in C^\infty_c (\R^2\times ]-\infty, \infty[)$ and
consider a smooth cut-off function $\vartheta$ of time only which vanishes identically on $]-\infty, \eps]$
and equals $1$ on $]\delta, \infty[$, where $0<\eps<\delta$. We know therefore that \eqref{eq:weak1}
holds for the test function $\psi \vartheta$, which implies that
\[
\int_0^\infty \int_{\R^2} \vartheta \left[\rho\dt \psi+ 
\rho v \cdot \nabla_x \psi\right] dx dt + \int_0^\delta \int_{\R^2} \vartheta' (t)
\rho (x,t) \psi (x,t) dx dt = 0\,  .
\]
Fix $\delta$ and choose a sequence of $\vartheta$ converging uniformly to the function
\[
\eta (t) = \left\{
\begin{array}{ll}
0 & \mbox{if $t\leq 0$}\\
1 & \mbox{if $t \geq \delta$}\\
\frac{t}{\delta} \qquad &\mbox{if $0\leq t \leq \delta$}
\end{array}
\right.
\]
and such that their derivatives $\vartheta'$ converge pointwise to $\frac{1}{\delta} \bm{1}_{]0, \delta[}$. We then conclude
\[
\int_0^\infty \int_{\R^2} \eta \left[\rho\dt \psi+ 
\rho v \cdot \nabla_x \psi\right] dx dt + \frac{1}{\delta} \int_0^\delta \int_{\R^2} 
\rho (x,t) \psi (x,t) dx dt = 0\,  .
\]
Letting $\delta\downarrow 0$ we conclude \eqref{eq:weak1}. 

The remaining conditions \eqref{eq:weak2} and \eqref{eq:admissibility condition} are achieved with analogous
arguments, which we leave to the reader.
\end{proof}

\section{Proof of Lemma \ref{l:ci}}\label{s:CI}

\subsection{Functional set-up} We define $X_0$ to be the space of 
$(\underline{v}, \underline{u})\in C^\infty_c (\Omega, \R^2\times
\Sym_0^{2\times 2})$ which satisfy (ii) and the pointwise inequality 
$(\tilde{v} + \underline{v})\otimes (\tilde{v} + \underline{v}) - (\tilde{u} + \underline{u}) < \frac{C}{2} \id$.
We then take the closure $X$ of $X_0$ in the $L^\infty$ weak$^\star$ topology and recall that, since $X$ is a bounded (weakly$^\star$)
closed subset of $L^\infty$ such topology is metrizable on $X$, giving a complete metric space $(X,d)$. 
Observe that any element in $X$ satisfies (i) and (ii)
and we want to show that on a residual set (in the sense of Baire category) (iii) holds. 
We then define for any $N\in \N\setminus \{0\}$ the map $I_N$ as follows: to $(\underline{v}, \underline{u})$ we associate
the corresponding restrictions of these maps to $B_N (0) \times ]-N, N[$. We then consider $I_N$ as a map from $(X,d)$ to 
$Y$, where $Y$ is the space $L^\infty (B_N (0)\times ]-N, N[, \R^2\times \Sym_0^{2\times 2}$) endowed with the {\em strong}
$L^2$ topology. Arguing as in \cite[Lemma 4.5]{dls1} it is easily seen that $I_N$ is a Baire-1 map and hence, from
a classical theorem in Baire category, its points of continuity are a residual set in $X$. We claim that
\begin{itemize}
 \item[(Con)] if $(\underline{v},
\underline{u})$ is a point of continuity of $I_N$, then (iii) holds a.e. on $B_N (0)\times ]-N, N[$.
\end{itemize}
(Con) implies then (iii) for those maps at which {\em all} $I_N$ are continuous (which is also a residual set). 

The proof of (Con) is achieved as in \cite[Lemma 4.6]{dls1} showing that:
\begin{itemize}
\item[(Cl)] If $(\underline{v}, \underline{u})\in X_0$, then  
there is a sequence
$(v_k, u_k)\subset X_0$ converging weakly$^\star$ to $(\underline{v},\underline{u})$ for which 
\[
\liminf_k \|\tilde{v} + v_k\|_{L^2 (\Gamma)} \geq 
\|\tilde{v} + \underline{v}\|^2_{L^2 (\Gamma)} + \beta \left(C |\Gamma| - \|\tilde{v} + 
\underline{v}\|^2_{L^2 (\Gamma)}\right)^2\, ,
\]
where $\Gamma= B_N (0)\times ]-N, N[$ and $\beta$ depends only on $\Gamma$. 
\end{itemize}

Indeed
assuming that (Cl) holds, fix then a point
$(\underline{v}, \underline{u})\in X$ where $I_N$ is continuous and assume by contradiction
that (iii) does not hold on $\Gamma$. By definition of $X$ there is a sequence 
$(\underline{v}_k, \underline{u}_k)\subset X_0$ converging weakly$^\star$ to 
$(\underline{v}, \underline{u})$. Since the latter
is a point of continuity for $I_N$, we then have that $\underline{v}_k \to \underline{v}$ strongly in $L^2 (\Gamma)$. 
We apply (Cl) to each $(\underline{v}_k, \underline{u}_k)$ and find a sequence $\{(v_{k,j}, u_{k,j})\}$ such that
\[
\liminf_j \|\tilde{v} + v_{k,j}\|_{L^2 (\Gamma)} \geq 
\|\tilde{v} + \underline{v}_k\|^2_{L^2 (\Gamma)} + \beta \left(C |\Gamma| - \|\tilde{v} + 
\underline{v}_k\|^2_{L^2 (\Gamma)}\right)^2\, 
\]
and $(v_{k,j}, u_{k,j})\rightharpoonup^\star (\underline{v}_k, \underline{u}_k)$.
A standard diagonal argument then allows to conclude the existence of a sequence $(v_{k, j(k)}, u_{k, j(k)})$
which converges weakly$^\star$ to $(\underline{v}, \underline{u})$ and such that
\[
 \liminf_k \|\tilde{v} + v_{k,j (k)}\|_{L^2 (\Gamma)} \geq 
\|\tilde{v} + \underline{v}\|^2_{L^2 (\Gamma)} + \beta \left(C |\Gamma| - \|\tilde{v} + 
\underline{v}\|^2_{L^2 (\Gamma)}\right)^2 > \|\tilde{v} + \underline{v}\|^2_{L^2 (\Gamma)}\, .
\]
However this contradicts the assumption that $(\underline{v}, \underline{u})$ is a point of continuity
for $I_N$.

In order to construct the sequence of (Cl) we appeal to the following
Proposition and Lemma.

\begin{proposition}[Localized plane waves]\label{p:local_waves}
Consider a segment $\sigma = [- p, p]\subset \R^2\times \Sym_0^{2\times 2}$, where $p =\lambda [(a, a\otimes a)
- (b, b\otimes b)]$ for some $\lambda>0$ and $a \neq \pm b$ with $|a|=|b| = \sqrt{C}$. Then there exists a pair 
$(v, u)\in C^\infty_c (B_1 (0)\times ]-1,1[)$ which solves 
\begin{equation}\label{eq:linear_PDE}
\left\{\begin{array}{l}
\partial_t v + {\rm div}_x u = 0\\ \\
{\rm div}_x v = 0
\end{array}\right.
\end{equation}
and such that
\begin{itemize}
\item[(i)] The image of $(v,u)$ is contained in an $\eps$-neighborhood of $\sigma$ and $\int (v,u)\, dx\, dt = 0$;
\item[(ii)] $\int |v(x,t)|\, dx\, dt \geq \alpha \lambda |b-a|$, where $\alpha$ is a positive 
constant depending only on $C$.
\end{itemize}
\end{proposition}

In order to the state the next lemma, it is convenient to introduce the following notation.

\begin{definition}\label{d:K}
Let $C>0$ be the positive constant of Lemma \ref{l:ci}. We let $\mathcal{U}$ be the subset of $\R^2\times
\Sym^{2\times 2}_0$ consisting of those pairs $(a,A)$ such that $a\otimes a - A < \frac{C}{2} {\rm Id}$. 
\end{definition}

\begin{lemma}[Geometric lemma]\label{l:geometric}
There exists a geometric constant $c_0$ with the following property. Assume $(a, A)\in \mathcal{U}$.
Then there is a segment $\sigma$ as in Proposition \ref{p:local_waves} with $(a, A) + \sigma \subset \mathcal{U}$
and $\lambda |b-a| \geq c_0 (C- |a|^2)$. 
\end{lemma}

We are now ready to prove (Cl). 
Let $(\underline{v}, \underline{u})\in X_0$. Consider any point $(x_0, t_0)\in \Gamma$ and observe that $(\tilde{v}, \tilde{u}) + 
(\underline{v}, \underline{u})$ takes values in $\mathcal{U}$. Let therefore $\sigma$ be as 
in Lemma \ref{l:geometric} when 
$(a, A)= (\tilde{v}, \tilde{u}) + (\underline{v} (x_0, t_0), \underline{u} (x_0, t_0))$
and choose $r>0$ so that $(\tilde{v}, \tilde{u}) + (\underline{v} (x,t), \underline{u} (x,t))
+ \sigma \subset \mathcal{U}$ for any $(x,t)\in B_r (x_0) \times ]t_0-r, t_0+r[$. For any $\eps>0$
consider a pair $(v,u)$ as in Proposition \ref{p:local_waves} and define $(v_{x_0, t_0, r}, u_{x_0, t_0, r})(x,t) :=
(v,u) \left(\frac{x-x_0}{r}, \frac{t-t_0}{r}\right)$. Observe that $(\underline{v}, \underline{u}) +
(v_{x_0, t_0, r}, u_{x_0, t_0, r})\in X_0$ provided $\varepsilon$ is sufficiently small, and moreover 
\begin{equation}\label{eq:rescaled_bound}
\int |v_{x_0, t_0, r}| \geq c_0\alpha\lambda (C- |\tilde{v} + \underline{v} (x_0, t_0)|^2) r^3\, .
\end{equation}

By continuity there exists $r_0$ such that the conclusion above holds for every $r<r_0$ and every
$(x,t)$ with $B_r (x)\times ]t-r, t+r[\subset \Gamma$.
Fix now $k\in \N$ with $\frac{1}{k}< r_0$. Set $r := \frac{1}{k}$ and find a finite number of points
$(x_j, t_j)$ such that the sets $B_r (x_j)\times ]t_j -r, t_j +r[$ are pairwise disjoint, contained
in $\Gamma$ and satisfy 
\begin{equation}\label{eq:Riemann_sum}
\sum_j \left(C- |\tilde{v} + \underline{v} (x_j, t_j)|^2\right) r^3 \geq \bar{c} 
\left(C |\Gamma| - \int_\Gamma |\tilde{v} + \underline{v} (x,t)|^2\, dx\, dt\right)\, ,
\end{equation}
where $\bar{c}$ is a suitable geometric constant. We then define
\[
(v_k, u_k):= (\underline{v}, \underline{u}) + \sum_j (v_{x_j, t_j, r}, u_{x_j, t_j, r}) \, .
\]
Since the supports of the $(v_{x_j, t_j, r}, u_{x_j, t_j, r})$ are pairwise disjoint, 
$(v_k, u_k)$ belongs to $X_0$ as well. Moreover, using the property that 
$\int (v_{x_j, t_j, r}, u_{x_j, t_j, r}) = 0$, it is immediate to check that $(v_k, u_k) \rightharpoonup^\star 
(\underline{v}, \underline{u})$ in $L^\infty$. On the other hand it also follows from
\eqref{eq:rescaled_bound} and \eqref{eq:Riemann_sum} that
\[
\|v_k -\underline{v}\|_{L^1 (\Gamma)} \geq c_1 \left(C |\Gamma| - \int_\Gamma |\tilde{v} + \underline{v}|^2\right)
\]
where the constant $c_1$ is only geometric. Using the weak$^\star$ convergence of $(v_k, u_k)$ 
to $(\underline{v}, \underline{u})$ we can
then conclude
\begin{align*}
\liminf_k \|\tilde{v} + v_k\|^2_{L^2 (\Gamma)} &= \|\tilde{v}+ \underline{v}\|_{L^2 (\Gamma)}^2 
+ \liminf_k \|v_k - \underline{v}\|^2\\  
&\geq \|\tilde{v}+ \underline{v}\|_{L^2 (\Gamma)}^2 
+ |\Gamma| \left(\liminf_k \|v_k-\underline{v}\|_{L^1}\right)^2\nonumber\\
&\geq \|\tilde{v}+ \underline{v}\|_{L^2 (\Gamma)}^2 + c_1^2 \|\Gamma\|
\left(C |\Gamma| - \int_\Gamma |\tilde{v} + \underline{v}|^2\right)^2\, ,
\end{align*}
which concludes the proof of the claim (Cl).

\subsection{Proof of Proposition \ref{p:local_waves} and of Lemma \ref{l:geometric}}

\begin{proof}[Proof of Proposition \ref{p:local_waves}]
Consider the $3\times 3$ matrices
\[
U_a =
\left(
\begin{array}{ll}
a \otimes a & a\\
a & 0 
\end{array}
\right) \qquad \mbox{and} \qquad
U_b =
\left(
\begin{array}{ll}
b \otimes b & b\\
b & 0 
\end{array}\right)
\]
Apply \cite[Proposition 4]{dls2} with $n=2$ to $U_a$ and $U_b$ and let $A (\partial)$ be the corresponding 
linear differential
operator and $\eta\in \R^2_x\times \R_t$ the corresponding vector. Let $\varphi$ be a cut-off function
which is identically equal to $1$ in $B_{1/2} (0)\times ]-\frac{1}{2}, \frac{1}{2}[$, 
is compactly supported in $B_1 (0)\times ]-1,1[$ and takes values in $[-1,1]$. For $N$ very large, whose choice will
be specified later, we consider the function
\[
 \phi (x,t) = - \lambda N^{-3} \sin (N \eta \cdot (x,t)) \varphi (x,t) =: \kappa (x,t) \varphi (x,t)\, 
\]
and we let $U (x,t) := A (\partial) (\phi)$. According to \cite[Proposition 4]{dls2}, $U: \R^2\times \R\to
\Sym^{3\times 3}$ is divergence free and trace-free and moreover $U_{33}=0$. Note also that $\int_{B_1 (0)\times ]-1,1[} U(x,t) = 0$. Define
\[
v (x,t) := (U_{31} (x,t), U_{32} (x,t))
\qquad
u (x,t) :=
\left(
\begin{array}{ll}
U_{11} (x,t) & U_{12} (x,t)\\
U_{21} (x,t) & U_{22} (x,t)
\end{array}\right)\, .
\]
It then follows easily that $(v,u)$ satisfies \eqref{eq:linear_PDE} and that it is supported in
$B_1 (0)\times ]-1,1[$. Also, since $A (\partial)$ is a 3rd order homogeneous linear differential
operator with constant coefficients, 
$\|U - \varphi A (\partial) (\kappa)\|_0 \leq C \lambda N^{-1}$, where $C$ depends only on the cut-off
function $\varphi$: in particular we can assume that $\|U - \varphi A (\partial) (\kappa)\|_0 < \varepsilon$.
On the other hand \cite[Proposition 4]{dls2} clearly implies that
\[
\varphi A (\partial) (\kappa) = \lambda (U_a- U_b) \varphi \cos (N (x,t)\cdot \eta)\, .
\]
we therefore conclude that $U$ takes values in an $\eps$-neighborhood of the segment $[-\lambda (U_a - U_b),
\lambda (U_a-U_b)]$. This obviously implies that 
$(v,u)$ takes values in an $\eps$-neighborhood of the segment $\sigma$. Finally, Let $B^3_{1/2}$ be the 
$3$-dimensional space-time ball in $\R^2 \times \R$, centered at $0$ and with radius $\frac{1}{2}$. Observe that
\begin{align*}
\int |v (x,t)| &\geq \int_{B^3_{1/2}} \lambda |a-b| |\cos (N (x,t) \cdot \eta)|\, dx\, dt
= \lambda |a-b| \int_{B^3_{1/2}} |\cos (N t |\eta|)|\, dx\, dt\, .  
\end{align*}
Moreover,
\[
\lim_{N\uparrow \infty} \int_{B^3_{1/2}} |\cos (N t |\eta|)|\, dx\, dt = \bar\alpha
\]
for some positive geometric constant $\bar \alpha$. 
\end{proof}

\begin{proof}[Proof of Lemma \ref{l:geometric}]
Consider the set 
\[
K_{\sqrt{C}}:= \left\{ (v,u)\in \R^2\times \Sym^{2\times 2}_0 : u = v\otimes v - \frac{C}{2} {\rm Id}\, ,
|v|^2 = C\right\}\, . 
\]
It then follows from \cite[Lemma 3]{dls2} that $\mathcal{U}$ is the interior of 
the convex hull of $K_{\sqrt{C}}$.
The existence of the claimed segment $\sigma$ is then a corollary of \cite[Lemma 6]{dls2},
since $\lambda |b-a|$ is indeed comparable (up to a geometric constant) to the length of $\sigma$. 
\end{proof}

\section{A set of algebraic identities and inequalities}\label{s:algebra}

In this paper we actually look at fan subsolutions with a fan partition consisting of only three
sets, namely $P_-, P_1$ and $P_+$. 

We introduce therefore the real numbers 
$\alpha, \beta, \gamma, \delta, v_{-1}, v_{-2}, v_{+1}, v_{+2}$ such hat
\begin{align} 
v_1 &= (\alpha, \beta),\label{eq:v1}\\
v_- &= (v_{-1}, v_{-2})\\
v_+ &= (v_{+1}, v_{+2})\\
u_1 &=\left( \begin{array}{cc}
    \gamma & \delta \\
    \delta & -\gamma\\
    \end{array} \right)\, .\label{eq:u1}
\end{align}

\begin{proposition}\label{p:algebra}
Let $N=1$ and $P_-, P_1, P_+$ be a fan partition as in Definition \ref{d:fan}. The constants 
$v_1, v_-, v_+, u_1, \rho_-, \rho_+, \rho_1$ as in \eqref{eq:v1}-\eqref{eq:u1} define an admissible
fan subsolution as in Definitions \ref{d:subs}-\ref{d:admiss} if and only if the following
identities and inequalities hold:
\begin{itemize}
\item Rankine-Hugoniot conditions on the left interface:
\begin{align}
&\nu_- (\rho_- - \rho_1) \, =\,  \rho_- v_{-2} -\rho_1  \beta \label{eq:cont_left}  \\
&\nu_- (\rho_- v_{-1}- \rho_1 \alpha) \, = \, \rho_- v_{-1} v_{-2}- \rho_1 \delta  \label{eq:mom_1_left}\\
&\nu_- (\rho_- v_{-2}- \rho_1 \beta) \, = \,  
\rho_- v_{-2}^2 + \rho_1 \gamma +p (\rho_-)-p (\rho_1) - \rho_1 \frac{C_1}{2}\, ;\label{eq:mom_2_left}
\end{align}
\item Rankine-Hugoniot conditions on the right interface:
\begin{align}
&\nu_+ (\rho_1-\rho_+ ) \, =\,  \rho_1  \beta - \rho_+ v_{+2} \label{eq:cont_right}\\
&\nu_+ (\rho_1 \alpha- \rho_+ v_{+1}) \, = \, \rho_1 \delta - \rho_+ v_{+1} v_{+2} \label{eq:mom_1_right}\\
&\nu_+ (\rho_1 \beta- \rho_+ v_{+2}) \, = \, - \rho_1 \gamma - \rho_+ v_{+2}^2 +p (\rho_1) -p (\rho_+) 
+ \rho_1 \frac{C_1}{2}\, ;\label{eq:mom_2_right}
\end{align}
\item Subsolution condition:
\begin{align}
 &\alpha^2 +\beta^2 < C_1 \label{eq:sub_trace}\\
& \left( \frac{C_1}{2} -{\alpha}^2 +\gamma \right) \left( \frac{C_1}{2} -{\beta}^2 -\gamma \right) - 
\left( \delta - \alpha \beta \right)^2 >0\, ;\label{eq:sub_det}
\end{align}
\item Admissibility condition on the left interface:
\begin{align}
& \nu_-(\rho_- \varepsilon(\rho_-)- \rho_1 \varepsilon( \rho_1))+\nu_- 
\left(\rho_- \frac{\abs{v_-}^2}{2}- \rho_1 \frac{C_1}{2}\right)\nonumber\\
\leq & \left[(\rho_- \varepsilon(\rho_-)+ p(\rho_-)) v_{-2}- 
( \rho_1 \varepsilon( \rho_1)+ p(\rho_1)) \beta \right] 
+ \left( \rho_- v_{-2} \frac{\abs{v_-}^2}{2}- \rho_1 \beta \frac{C_1}{2}\right)\, ;\label{eq:E_left}
\end{align}
\item Admissibility condition on the right interface:
\begin{align}
&\nu_+(\rho_1 \varepsilon( \rho_1)- \rho_+ \varepsilon(\rho_+))+\nu_+ 
\left( \rho_1 \frac{C_1}{2}- \rho_+ \frac{\abs{v_+}^2}{2}\right)\nonumber\\
\leq &\left[ ( \rho_1 \varepsilon( \rho_1)+ p(\rho_1)) \beta- (\rho_+ \varepsilon(\rho_+)+ p(\rho_+)) v_{+2}\right] 
+ \left( \rho_1 \beta \frac{C_1}{2}- \rho_+ v_{+2} \frac{\abs{v_+}^2}{2}\right)\, .\label{eq:E_right}
\end{align}
\end{itemize}
\end{proposition}

\begin{figure}[htbp]
\begin{center}
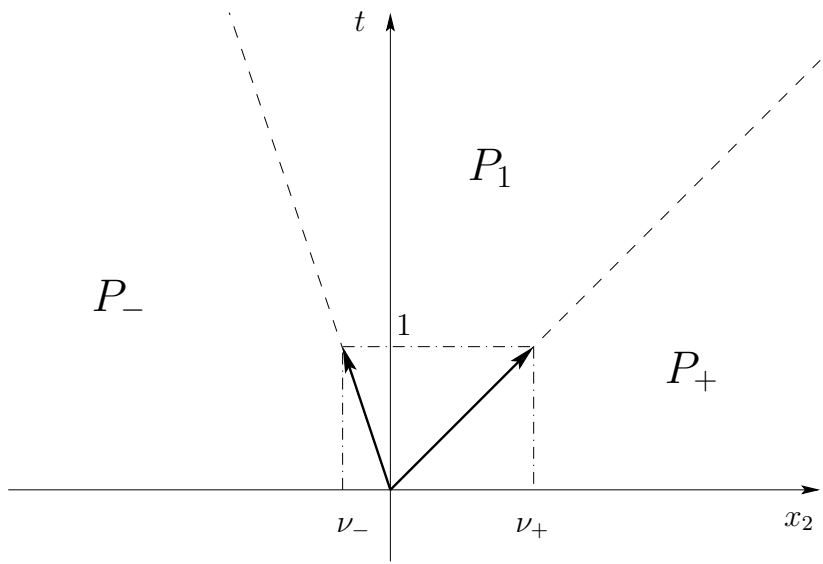
\caption{The fan partition in three regions.}
\label{fig:subsolution}
\end{center}
\end{figure}

\begin{proof} Observe that the triple $(\overline{\rho}, \overline{v}, \overline{u})$ does not
depend on the variable $x_1$. We will therefore consider it as a map defined on the $t, x_2$ plane.
The various conditions and inequalities follow from straightforward computations,
recalling that the maps $\overline{\rho}, \overline{v}$ and $\overline{u}$ are constant in the regions
$P_-$, $P_1$ and $P_+$ shown in Figure \ref{fig:subsolution}. In particular
\begin{itemize}
\item The identities \eqref{eq:cont_left} and \eqref{eq:cont_right} are equivalent to 
the continuity equation \eqref{eq:continuity}, in particular they derive from the 
corresponding ``Rankine-Hugoniot'' type conditions at the interfaces between $P^-$ and $P_1$ 
(the {\em left interface}) and $P_1$ and $P_+$ (the {\em right interface}), respectively.
\item The identitities \eqref{eq:mom_1_left} and \eqref{eq:mom_1_right} are the Rankine-Hugoniot
conditions at the left and right interfaces resulting from the first component of the momentum
equation \eqref{eq:momentum}; similarly \eqref{eq:mom_2_left} and \eqref{eq:mom_2_right}
correspond to the Rankine-Hugoniot conditions at the left and right interfaces for the second
component of the momentum equation \eqref{eq:momentum}.
\item The inequalities \eqref{eq:sub_trace} and \eqref{eq:sub_det} are derived applying the
usual criterion that the matrix 
\begin{equation}
M:= \frac{C_1}{2} \id - v_1\otimes v_1 + u_1
\end{equation}
is positive definite if and only if $\tr M$ and $\det M$ are both positive.
\item Finally, the conditions \eqref{eq:E_left} and \eqref{eq:E_right} derive from the admissibility
condition \eqref{eq:admissible subsolution}, again considering, respectively, the corresponding
inequalities at the left and right interfaces. 
\end{itemize}
\end{proof}

\section{First method: data generated by compression waves for $p (\rho) = \rho^2$}\label{s:rho^2_1}

In this section we show how to find solutions of the algebraic constraints in Proposition \ref{p:algebra}
when $p (\rho) = \rho^2$ with pairs $(\rho_\pm, v_\pm)$ which can be connected by a compression wave,
thereby showing Theorem \ref{t:main}. We start by recalling the following fact, which can be easily derived
using (by now) standard theory of hyperbolic conservation laws in one space dimension. 

\begin{lemma}\label{l:comp_wave}
Let $0<\rho_-<\rho_+$, $v_+ = (-\frac{1}{\rho_+}, 0)$ and $v_- = (- \frac{1}{\rho_+}, 2 \sqrt{2} (\sqrt{\rho_+}
- \sqrt{\rho_-}))$. Then there is a pair $(\rho, v)\in W^{1,\infty}_{loc} \cap L^\infty 
(\R^2 \times ]-\infty, 0[, \R^+\times \R^2)$ such that
\begin{itemize}
 \item[(i)] $\rho_+ \geq \rho \geq \rho_- >0$;
\item[(ii)] The pair solves the hyperbolic system
\begin{equation}\label{eq:Euler_2}
\left\{\begin{array}{l}
\partial_t \rho + {\rm div}_x (\rho v) \;=\; 0\\
\partial_t (\rho v) + {\rm div}_x \left(\rho v\otimes v \right) + \nabla_x [ p(\rho)]\;=\; 0
\end{array}\right.
\end{equation}
with $p(\rho) = \rho^2$ in the classical sense (pointwise a.e. and distributionally);
\item[(iii)] for $t\uparrow 0$ the pair $(\rho (\cdot, t), v (\cdot , t))$ converges pointwise a.e. to 
$(\rho^0, v^0)$ as in \eqref{eq:R_data};
\item[(iv)] $(\rho (\cdot ,t), v (\cdot , t))\in W^{1,\infty}$ for every $t<0$.
\end{itemize}
  \end{lemma}

As already mentioned, the proof is a very standard application of the one-dimensional theory 
for the so-called Riemann problem. However, we give the details for the reader's convenience.

\begin{proof} We look for solutions $(\rho, v)$ with the claimed properties which are independent of the $x_1$
variable. Moreover we observe that, since we will produce classical $W^{1, \infty}_{loc}$ solutions,
the admissibility condition \eqref{eq:admissibility condition}
will be automatically satisfied as an equality because 
\[
\left(\rho \eps (\rho) + \frac{|v|^2}{2} \rho,
\left(\rho \eps (\rho) + \frac{|v|^2}{2} \rho + p (\rho)\right) v\right)
\] 
is an entropy-entropy flux pair for the
system \eqref{eq:Euler_2} (cf. \cite[Sections 3.2, 3.3.6, 4.1]{da}). We then introduce 
the unknowns 
\[
(m_1 (x_2, t), m_2 (x_2, t)) = m (x_2, t) := v (x_2, t) \rho (x_2, t)
\] 
and hence rewrite the system as 
\begin{equation}\label{eq:explicit system}
\left\{\begin{array}{l}
\partial_t \rho + \dx m_2 \;=\; 0\\
\partial_t m_1 + \dx \left(\frac{m_1m_2 }{\rho}\right)\;=\; 0\\
\partial_t m_2 + \dx \left(\frac{{m_2}^2}{\rho}+\rho^2\right)\;=\; 0
\end{array}\right.
\end{equation}
Observe that if $(\rho, m)$ is a solution of \eqref{eq:explicit system} then so is
\[
(\tilde{\rho} (x_2, t), \tilde{m} (x_2, t)) := (\rho (-x_2, -t), m (-x_2, -t))\, .
\] 
Moreover,
if $(\rho, m)$ is locally Lipschitz and hence satisfies the admissibility condition with {\em equality},
so does $(\tilde{\rho}, \tilde{m})$. We have therefore reduced ourselves to finding classical
$W^{1,\infty}_{loc}$ solutions on $\R\times ]0, \infty[$ of \eqref{eq:explicit system} with initial data
\begin{equation}\label{eq:initial density new}
\rho_0(x):= 
\begin{cases} \rho_R & \text{if $x_2>0$,}
\\
\rho_L &\text{if $x_2< 0$,}
\end{cases}
\end{equation}
and 
\begin{equation}\label{eq:initial momentum}
m_0(x):= 
\begin{cases} m_R:=\left(-\frac{\rho_R}{\rho_L},2\sqrt{2} \rho_R(\sqrt{\rho_L}-\sqrt{\rho_R})\right) & \text{if $x_2>0$,}
\\
m_L:=\left(-1,0\right) &\text{if $x_2< 0$,}
\end{cases}
\end{equation}
where $\rho_+ = \rho_L > \rho_R = \rho_- >0$, $m_L = v_+ \rho_+$ and $m_R = v_-\rho_-$. 

The problem amounts now in showing that, under our assumptions, there is a classical rarefaction wave
solving, forward in time, the system \eqref{eq:explicit system} with initial data $(\rho_0, m_0)$ as
in \eqref{eq:initial density new} and \eqref{eq:initial momentum}.
We set therefore $p (\rho) = \rho^2$ and we look for a locally Lipschitz self-similar solution 
$(\rho,m)$ to the Riemann problem 
\eqref{eq:explicit system}-\eqref{eq:initial density new}-\eqref{eq:initial momentum}:  
\begin{equation} \label{eq:self-similar solution}
 (\rho, m)(x_2,t)= (R,M)\left(\frac{x_2}{t}\right), \quad -\infty <x_2<\infty, \quad 0<t<\infty\, .
\end{equation}
Thus $(R,M)$ are locally Lipschitz functions on $(- \infty, \infty)$
which satisfy the ordinary differential equations
\begin{align*}
&\frac{d}{d\xi} \left[M_2(\xi)-\xi R(\xi)\right]+ R(\xi)=0\\
&\frac{d}{d\xi} \left[ \frac{M_1(\xi)M_2(\xi)}{R(\xi)}- \xi M_1(\xi)\right] +M_1(\xi)=0 \\
&\frac{d}{d\xi} \left[ \frac{M_2(\xi)^2}{R(\xi)}+p(R(\xi))- \xi M_2(\xi)\right] +M_2(\xi)=0\, . 
\end{align*}
Before analyzing our specific Riemann problem, we review some general notions for system 
\eqref{eq:explicit system} (referring the reader to the monographs \cite{da} and \cite{se}).
If we introduce the state vector $U:= (\rho, m_1, m_2)$, we can recast the system \eqref{eq:explicit system} in the general form
$$\dt U + \dx F(U)=0,$$
where 
$$
F(U):=\left( \begin{array}{c}
    m_2 \\
    \frac{m_1 m_2}{\rho}\\
    \frac{{m_2}^2}{\rho}+p(\rho)\\
    \end{array} \right).
$$
By definition (cf. \cite{da}) the system \eqref{eq:explicit system} is hyperbolic since the Jacobian matrix $DF(U)$
$$
DF(U)=\left( \begin{array}{ccc}
    0 &0 &1 \\
    \frac{-m_1 m_2}{\rho^2} &\frac{m_2}{\rho} &\frac{m_1}{\rho}\\
    \frac{-{m_2}^2}{\rho^2}+p'(\rho) &0 &\frac{2m_2}{\rho}\\
    \end{array} \right)
$$
has real eigenvalues 
\begin{equation} \label{eq:eigenvalues}
 \lambda_1= \frac{m_2}{\rho}-\sqrt{p'(\rho)}, \qquad \lambda_2=\frac{m_2}{\rho}, \qquad \lambda_3= \frac{m_2}{\rho}+\sqrt{p'(\rho)}
\end{equation}
and $3$ linearly independent eigenvectors 
\begin{equation} \label{eq:eigenvectors}
 R_1=\left( \begin{array}{c}
    1\\
    \frac{m_1}{\rho}\\
    \frac{m_2}{\rho}-\sqrt{p'(\rho)}\\
    \end{array} \right), \quad 
R_2=\left( \begin{array}{c}
    0\\
    1\\
    0\\
    \end{array} \right), \quad 
 R_3=\left( \begin{array}{c}
    1\\
    \frac{m_1}{\rho}\\
    \frac{m_2}{\rho}+\sqrt{p'(\rho)}\\
    \end{array} \right).
\end{equation}
The eigenvalue $\lambda_i$ of $DF$, $i=1,2,3$, is called the $i$-\textit{characteristic speed} of the system \eqref{eq:explicit system}.
On the part of the state space of our interest, with $\rho>0$, the system \eqref{eq:explicit system} is indeed strictly hyperbolic.
Finally, one can easily verify that the functions
\begin{equation}\label{eq:invariants}
w_3=\frac{m_2}{\rho} +\int_0^\rho \frac{\sqrt{p'(\tau)}}{\tau} d\tau, \qquad w_2=\frac{m_1}{\rho}, 
\qquad w_1=\frac{m_2}{\rho} -\int_0^\rho \frac{\sqrt{p'(\tau)}}{\tau} d\tau
\end{equation}
are, respectively, ($1$- and $2$-), ($1$- and $3$-), ($2$- and $3$-) Riemann invariants of the system \eqref{eq:explicit system} (for the relevant definitions see \cite{da}).

In order to characterize rarefaction waves of the reduced system 
\eqref{eq:explicit system}, we can refer to Theorem $7.6.6$ from \cite{da}: every $i$- Riemann invariant is constant 
along any $i$- rarefaction wave curve of the system \eqref{eq:explicit system} and conversely 
the $i$- rarefaction wave curve, through a state $(\overline{\rho}, \overline{m})$ of genuine nonlinearity of the $i$- characteristic family,
is determined implicitly by the system of equations $w_i(\rho, m)=w_i(\overline{\rho}, \overline{m})$ for every $i$- Riemann invariant $w_i$.
As an application of this theorem, we obtain that $(\rho_R,m_R)$ lies on the $1$- rarefaction wave through $(\rho_L, m_L)$.
Indeed, the $1$- rarefaction wave of the system \eqref{eq:explicit system} through the point $(\rho_L,m_L)$ is
 determined in terms of the Riemann invariants $w_3$ and $w_2$ by the equations
\begin{equation} \label{eq:rarefaction waves}
 m_1= -\frac{\rho}{\rho_L}, \qquad m_2=\rho \int_{\rho}^{\rho_L} \frac{\sqrt{p'(\tau)}}{\tau} d\tau,
\end{equation}
with $\rho<\rho_L$. In the case of pressure law $p(\rho)=\rho^2$, the 
equations \eqref{eq:rarefaction waves} read as
\begin{equation} \label{eq:rarefaction waves'}
 m_1= -\frac{\rho}{\rho_L}, \qquad m_2= 2\sqrt{2} \rho\left(\sqrt{\rho_L}-\sqrt{\rho}\right).
\end{equation}
Clearly, the constant state $(\rho_R, m_R)$, as defined by \eqref{eq:initial density new}--\eqref{eq:initial 
momentum}, satisfies the equations \eqref{eq:rarefaction waves'}.
Since, according to Theorem 7.6.5 in \cite{da}, there exists a unique $1$-rarefaction wave 
through $(\rho_L, m_L)$, we have shown the existence of our desired self-similar locally Lipschitz solution.

Observe that, by construction, $\rho_+ = \rho_L \geq \rho \geq \rho_R = \rho_- > 0$, thereby showing (i).
The claim (iv) follows easily because there exists a constant $C>0$ such that, 
for every positive time $t$, the pair $(\rho, m)$ takes the constant value $(\rho_R, m_R)$ for $x_2 \geq C t$
and $(\rho_L, m_L)$ for $x_2 \leq -C t$. 
\end{proof}

We next show the existence of a solution of the algebraic constraints of Proposition \ref{p:algebra}
such that in addition $(\rho_\pm, v_\pm)$ satisfy the identities of Lemma \ref{l:comp_wave}.

\begin{lemma}\label{l:exp_sol}
Let $p (\rho ) = \rho^2$. There exist $\rho_\pm, v_\pm$ satisfying the assumptions of Lemma
\ref{l:comp_wave} and $\rho_1, C_1, v_1, u_1, \nu_\pm$ satisfying the algebraic identities and inequalities
\eqref{eq:cont_left}-\eqref{eq:E_right}.  
\end{lemma}
\begin{proof} Taking into account that $p(\rho) = \rho^2$ and therefore
$\varepsilon(\rho) = \rho$, we substitute the identities of Lemma \ref{l:comp_wave} into the
unknowns of Proposition \ref{p:algebra} and reduce \eqref{eq:cont_right}-\eqref{eq:mom_2_right} to
\begin{align}
&\nu_+(\rho_1-\rho_+) = \rho_1\beta \label{eq:byebye1}\\
&\nu_+(\rho_1\alpha+1) = \rho_1\delta\label{eq:byebye2} \\
&\nu_+\rho_1\beta = -\rho_1\gamma + \rho_1^2 - \rho_+^2 + \rho_1\frac{C_1}{2}\, .
\end{align}
Similarly, we reduce \eqref{eq:cont_left}-\eqref{eq:mom_2_left} to
\begin{align}
&\nu_-(\rho_- - \rho_1) = 2\sqrt{2}\rho_-(\sqrt{\rho_+}-\sqrt{\rho_-}) - \rho_1\beta\\
&\nu_-(-\frac{\rho_-}{\rho_+} - \rho_1\alpha) = -2\sqrt{2}\frac{\rho_-}{\rho_+}(\sqrt{\rho_+}-\sqrt{\rho_-})
- \rho_1\delta \\
&\nu_-(2\sqrt{2}\rho_-(\sqrt{\rho_+}-\sqrt{\rho_-}) - \rho_1\beta) = 8\rho_-(\sqrt{\rho_+}-\sqrt{\rho_-})^2 
+ \rho_1\gamma + \rho_-^2- \rho_1^2 - \rho_1\frac{C_1}{2}\, .
\end{align}
The identities of Lemma \ref{l:comp_wave} do not influence the form of \eqref{eq:sub_trace}-\eqref{eq:sub_det}.
Instead, plugging them into \eqref{eq:E_left}-\eqref{eq:E_right} the latter are reduced to
\begin{align}
&\nu_- \left( \rho_-^2 - \rho_1^2 + \frac{\rho_-}{2 \rho_+^2} + 4 \rho_- (\sqrt{\rho_+}-\sqrt{\rho_-})^2
- \frac{C_1\rho_1}{2}\right)\nonumber\\
\leq\; & 
\sqrt{2} \rho_- (\sqrt{\rho_+}-\sqrt{\rho_-}) \left(4 \rho_- + \frac{1}{\rho_+^2}
+ 8 \rho_- (\sqrt{\rho_+}-\sqrt{\rho_-})^2\right) - 2 \rho_1^2 \beta - \frac{\beta C_1 \rho_1}{2}\\
&\nu_+ \left( \rho_1^2 - \rho_+^2 + \frac{C_1\rho_1}{2} - \frac{1}{2\rho_+}\right) \leq 2 \rho_1^2 \beta
+ \frac{C_1 \rho_1\beta}{2}\label{eq:byebye3}\, .
\end{align}
We next make the choice $\nu_+=\beta = \delta = 0$ and hence \eqref{eq:byebye1}, \eqref{eq:byebye2}
and \eqref{eq:byebye3}
are automatically satisfied. The remaining constraints above then become
\begin{align}
0 = &-\rho_1\gamma + \rho_1^2 - \rho_+^2 + \rho_1\frac{C_1}{2}\label{eq:surv1}\\
\nu_-(\rho_- - \rho_1) = &2\sqrt{2}\rho_-(\sqrt{\rho_+}-\sqrt{\rho_-})\label{eq:surv2}\\
\nu_-(-\frac{\rho_-}{\rho_+} - \rho_1\alpha) = &-2\sqrt{2}\frac{\rho_-}{\rho_+}(\sqrt{\rho_+}-\sqrt{\rho_-})
\label{eq:surv3}\\
\nu_-(2\sqrt{2}\rho_-(\sqrt{\rho_+}-\sqrt{\rho_-})) = &8\rho_-(\sqrt{\rho_+}-\sqrt{\rho_-})^2 
+ \rho_1\gamma + \rho_-^2- \rho_1^2 - \rho_1\frac{C_1}{2}\label{eq:surv4}
\end{align}
and
\begin{align}
&\nu_- \left( \rho_-^2 - \rho_1^2 + \frac{\rho_-}{2 \rho_+^2} + 4 \rho_- (\sqrt{\rho_+}-\sqrt{\rho_-})^2
- \frac{C_1\rho_1}{2}\right)\nonumber\\
\leq\; & 
\sqrt{2} \rho_- (\sqrt{\rho_+}-\sqrt{\rho_-}) \left(4 \rho_- + \frac{1}{\rho_+^2}
+ 8 \rho_- (\sqrt{\rho_+}-\sqrt{\rho_-})^2\right)\, .\label{eq:surv5}
\end{align}
Moreover, \eqref{eq:sub_trace} and \eqref{eq:sub_det} become
\begin{align}
\alpha^2\; &< C_1\label{eq:surv6}\\
 0\; &< \left(\frac{C_1}{2} - \alpha^2 + \gamma \right) \left(\frac{C_1}{2} - \gamma\right)\label{eq:surv7}\, .
\end{align}
Summarizing we are looking for real numbers $\nu_-<0, 0<\rho_-<\rho_+, \rho_1, \alpha, \gamma$ and $C_1$
satisfying the set of identities and inequalities \eqref{eq:surv1}-\eqref{eq:surv7}.

We next choose $ \rho_-=1 < 4=\rho_+$ and simplify further \eqref{eq:surv1}-\eqref{eq:surv5} as
\begin{align}
&\frac{C_1\rho_1}{2} + \rho_1^2 - \rho_1 \gamma - 16 = 0 \label{eq:surv2_1}\\
&\nu_- (1-\rho_1) = 2 \sqrt{2} \label{eq:surv2_2}\\
&\nu_- \left(\frac{1}{4} + \alpha \rho_1\right) = \frac{\sqrt{2}}{2}\label{eq:byebye11}\\
& 9 + \rho_1 \gamma - \rho_1^2 - \frac{C_1 \rho_1}{2} = 2 \sqrt{2} \nu_-\label{eq:byebye12}\\
& \nu_- \left( 5 + \frac{1}{32} - \rho_1^2 - \frac{C_1\rho_1}{2}\right) \leq \sqrt{2} \left(12 + \frac{1}{16}\right)
\, .\label{eq:surv2_5}
\end{align}
We now observe that \eqref{eq:surv2_2} and \eqref{eq:byebye11} imply $\alpha = -\frac{1}{4}$ and
\eqref{eq:surv2_1}-\eqref{eq:byebye12} imply $\nu_- = - \frac{7}{2\sqrt{2}}$. Therefore, our constraints
further simplify to looking for $\rho_1, \gamma, C_1$ such that
\begin{align}
\frac{1}{16} &< C_1 \label{eq:byebye22}\\
 0\; &< \left(\frac{C_1}{2} - \frac{1}{16} + \gamma \right) \left(\frac{C_1}{2} - \gamma\right)\label{eq:surv3_2}\\
0 &= \frac{C_1\rho_1}{2} + \rho_1^2 - \rho_1 \gamma - 16 \label{eq:surv3_3}\\
8 &= - 7 (1-\rho_1)\label{eq:byebye21}\\
48 + \frac{1}{4} &\geq -7 \left(5 + \frac{1}{32} -\rho_1^2 - \frac{C_1 \rho_1}{2}\right)\, .\label{eq:surv3_5}
\end{align}
From \eqref{eq:byebye21} we derive $\rho_1 = \frac{15}{7}$ and inserting this into \eqref{eq:surv3_3} we
infer $\frac{C_1}{2} - \gamma = \frac{559}{105}$. In turn this last identity reduces \eqref{eq:surv3_2}
to the inequality 
\begin{equation}\label{eq:surv4_1}
C_1 > \frac{1}{16} + \frac{559}{105}\, . 
\end{equation}
The remaining constraints \eqref{eq:surv3_3} and \eqref{eq:surv3_5} simplify to:
\begin{align}
\frac{C_1}{2} - \gamma &= \frac{559}{105}\label{eq:byebye31}\\
48 + \frac{1}{4} + 35 + \frac{7}{32} - \frac{225}{7} &\geq \frac{15 C_1}{2}\, .\label{eq:surv4_2}
\end{align}
We therefore see that $\gamma$ can be obtained from $C_1$ through \eqref{eq:byebye31}. Hence
the existence of the desired solution is equivalent to the existence of a $C_1$
satisfying \eqref{eq:surv4_1} and \eqref{eq:surv4_2}. Such $C_1$ exists if and only if
\[
\frac{15}{2} \left(\frac{1}{16} + \frac{559}{105}\right) < 48 + \frac{1}{4} + 35 + 
\frac{7}{32} - \frac{225}{7}\, ,
\]
which can be trivially checked to hold.
\end{proof}

Theorem \ref{t:main} and Corollary \ref{c:compression} now easily follow.

\begin{proof}[Proofs of Theorem \ref{t:main} and Corollary \ref{c:compression}]
Let $p (\rho) = \rho^2$ and consider the $\rho_\pm, v_\pm$ given by Lemma \ref{l:exp_sol}.
Applying Propositions \ref{p:algebra} and \ref{p:subs} we know that there are infinitely many
admissible solutions of \eqref{eq:Euler system}-\eqref{eq:R_data} as claimed in the Theorem.

Let now $(\rho_f, v_f)$ be any such solution and let $(\rho_b, v_b)$ be the locally
Lipschitz solutions of \eqref{eq:Euler system} given by Lemma \ref{l:comp_wave}. It is
straightforward to check that, if we define
\begin{equation}
(\rho, v) (x,t) := \left\{
\begin{array}{ll}
(\rho_f , v_f) (x,t) \qquad &\mbox{if $t\geq 0$}\\ \\
(\rho_b, v_b) (x,t) \qquad &\mbox{if $t\leq 0$}\, ,
\end{array}
\right.
\end{equation}
then the pair $(\rho, v)$ is a bounded admissible solution of \eqref{eq:Euler system} on the
entire space-time $\R^2\times \R$ with density bounded away from $0$. Moreover
$(\rho (\cdot, t), v (\cdot ,t))$ is a bounded Lipschitz function for every $t<0$.
In particular we can define $(\tilde{\rho}, \tilde{v}) (x,t) = (\rho, v) (x, t-1)$ and observe that,
no matter which of the infinitely many solutions $(\rho_f, v_f)$ given by Theorem \ref{t:main} we choose,
the corresponding $(\tilde{\rho}, \tilde{v})$ defined above is an admissible solution as in
Corollary \ref{c:compression} for the bounded and Lipschitz initial data 
$(\rho^0, v^0) = (\rho_b, v_b) (\cdot, -1)$.
\end{proof}

\begin{remark}\label{r:general pressure}
In fact, it is not difficult to see by a simple continuity argument that the conclusions of Theorem \ref{t:main} and Corollary \ref{c:compression} hold even for general pressure laws $p(\rho) = \rho^\gamma$ with $\gamma$ in some neighborhood of $2$.
\end{remark}

\section{Second method: further Riemann data for different pressures}\label{s:p_strana}

In this section we describe a second method for producing solutions to the algebraic set of
equations and inequalities of Proposition \ref{p:algebra}. Unlike the method given in the 
previous section, we do not know whether this one produces Riemann data generated by a compression
wave: we can only show that this is not the case for the ones which we have computed explicitely.  
Moreover we do not fix the pressure law but we exploit it as an extra degree of freedom.
On the other hand the reader can easily check that the method below gives a rather large set
of solutions (i.e. open) compared to the one of Lemma \ref{l:exp_sol} (where we do not know
whether one can perturb the choice $\nu_+=0$).

\begin{lemma}\label{l:p_strana}
Set $v_\pm = (\pm 1, 0)$. Then there exist $\nu_\pm$, $\rho_\pm$, $\rho_1$, $\alpha, \beta, \gamma, \delta$,
$C_1$ and a smooth pressure $p$ with $p' >0$ for which the algebraic identities and inequalities
\eqref{eq:cont_left}-\eqref{eq:E_right} are satisfied.  
\end{lemma}
\subsection{Part I of the proof of Lemma \ref{l:p_strana}: reduction of the admissibility conditions}
We rewrite the conditions \eqref{eq:cont_left}-\eqref{eq:mom_2_right}
\begin{align}
\nu_- (\rho_1 - \rho_-) \, &=\, \rho_1  \beta  \label{eq:st1}\\
\nu_- (\rho_- + \rho_1 \alpha) \, &=  \rho_1 \delta \label{eq:st2} \\
\rho_1 \frac{C_1}{2} - \rho_1 \gamma + p (\rho_1) - p(\rho_-)\, &=\, \nu_- \rho_1 \beta\label{eq:st3}\\
\nu_+ (\rho_1-\rho_+ ) \, &=\,  \rho_1  \beta \label{eq:st4}\\
\nu_+ (\rho_1 \alpha- \rho_+) \, &= \, \rho_1 \delta \label{eq:st5}\\
\rho_1 \frac{C_1}{2} - \rho_1 \gamma + p (\rho_1) - p(\rho_+)\, &=\, \nu_+ \rho_1 \beta\label{eq:st6}\, .
\end{align}
The conditions \eqref{eq:sub_trace} and \eqref{eq:sub_det} are not affected by our choice.
The conditions \eqref{eq:E_left} and \eqref{eq:E_right} become
\begin{align}
\nu_- \left(\rho_- \eps (\rho_-) - \rho_1 \eps (\rho_1) 
+ \frac{\rho_-}{2} - \rho_1 \frac{C_1}{2}\right) + \beta \left( \rho_1 \eps (\rho_1) + p (\rho_1)
+ \rho_1\frac{C_1}{2}\right)\, &\leq 0
\label{eq:st7}\\
\nu_+\left(\rho_1 \eps (\rho_1) - \rho_+ \eps (\rho_+) + \rho_1 \frac{C_1}{2}- \frac{\rho_+}{2}\right)
- \beta \left( \rho_1 \eps (\rho_1) + p (\rho_1) + \rho_1 \frac{C_1}{2}\right)\, &\leq 0 \label{eq:st8}\, .
\end{align}
Plugging \eqref{eq:st1} and \eqref{eq:st4} into, respectively, \eqref{eq:st7} and \eqref{eq:st8}
we obtain
\begin{align}
& \nu_- \left(\rho_- \eps (\rho_-) - \rho_1 \eps (\rho_1) - \rho_1 \frac{C_1-1}{2}\right) + \beta 
\left(\rho_1 \eps (\rho_1) + p (\rho_1) + \rho_1 \frac{C_1-1}{2}\right) \leq 0\label{eq:byebye41}\\
&\nu_+\left(\rho_1 \eps (\rho_1) - \rho_+ \eps (\rho_+) + \rho_1 \frac{C_1-1}{2}\right)
-\beta \left( \rho_1 \eps (\rho_1) + p (\rho_1) + \rho_1 \frac{C_1-1}{2}\right)\leq 0\, .\label{eq:byebye42}
\end{align}
We next rely on the following

\begin{lemma}\label{l:choice_of_pressure}
Assume that
\begin{align} 
& \nu_-<0<\nu_+\, ,\label{eq:nu1<nu2'}\\
& \rho_-<\rho_+\, .\label{eq:rho<rho}
\end{align}
Then, there exist pressure functions $p\in C^\infty([0,+\infty[)$ with $p'>0$ on $]0,+\infty[$ such that the 
admissibility conditions \eqref{eq:byebye41}-\eqref{eq:byebye42} for a subsolution are implied by 
the following system of inequalities:
\begin{align}
 &\left(p(\rho_+) -p(\rho_1)\right)\left(\rho_+ - \rho_1 \right) 
> \frac{C_1-1}{2} \rho_+ \rho_1 \label{eq:surv5_1} \\
&\left(p(\rho_1) -p(\rho_-)\right)\left(\rho_1 -\rho_- \right) 
> \frac{C_1-1}{2} \rho_- \rho_1. \label{eq:surv5_2}
\end{align}
\end{lemma}

\begin{proof} First, let us define $g(\rho):=\rho \varepsilon(\rho)$. In view of the relation 
$p(\rho)=\rho^2 \varepsilon'(\rho)$, we obtain
\begin{equation*} 
 g'(\rho)=\varepsilon(\rho) +\frac{p(\rho)}{\rho}.
\end{equation*}
Thus, by vitue of \eqref{eq:st1} and \eqref{eq:st4}, respectively, 
we can rewrite \eqref{eq:byebye41} and \eqref{eq:byebye42} as follows:
\begin{align} 
&\nu_- (g(\rho_-)-g(\rho_1))+\nu_- (\rho_1- \rho_-) g'(\rho_1)
-\nu_- \rho_- \frac{C_1-1}{2}\leq 0 \label{eq:inequality 2 new'}\\
&\nu_+ (g(\rho_1)-g(\rho_+))+\nu_+ (\rho_+ - \rho_1)g'(\rho_1)+\nu_+\rho_+ \frac{C_1-1}{2} \leq 0\, .
\label{eq:inequality 1 new'} 
\end{align}
From the hypothesis \eqref{eq:nu1<nu2'} we can further reduce \eqref{eq:inequality 2 new'}-\eqref{eq:inequality 1 new'} to
\begin{align} 
-&(g(\rho_1)-g(\rho_-))+ (\rho_1- \rho_-) g'(\rho_1)\geq \frac{C_1-1}{2} \rho_-  \label{eq:inequality 2 simplified'}\\
&(g(\rho_+)-g(\rho_1))- (\rho_+ - \rho_1)g'(\rho_1)\geq \frac{C_1-1}{2} \rho_+\, . \label{eq:inequality 1 simplified'}
\end{align}
Moreover, we observe from \eqref{eq:st1}-\eqref{eq:st4} that
$$\nu_+(\rho_+- \rho_1)=-\nu_-(\rho_1-\rho_-).$$
Hence, in view of \eqref{eq:nu1<nu2'}-\eqref{eq:rho<rho}, we must have
\begin{equation} \label{eq:case1'}
 \rho_-<\rho_1<\rho_+\, .
\end{equation}
Let us note that
$$\left(g(\sigma)-g(s)\right)-(\sigma-s) g'(s)= \int_s^\sigma \int_s^\tau g''(r)dr d\tau$$
for every $s<\sigma$. On the other hand, by simple algebra, we can compute $g''(r)= p'(r)/r$. Hence, the following equalities hold for every $s<\sigma$:
$$\left(g(\sigma)-g(s)\right)-(\sigma-s) g'(s)= \int_s^\sigma \int_s^\tau \frac{p'(r)}{r} dr d\tau$$
and
$$\left(g(s)-g(\sigma)\right)+(\sigma-s) g'(\sigma)= \int_s^\sigma \int_\tau^\sigma \frac{p'(r)}{r} dr d\tau.$$
As a consequence, and in view of \eqref{eq:case1'}, we can rewrite \eqref{eq:inequality 2 simplified'}
and \eqref{eq:inequality 1 simplified'} equivalently as
\begin{align}
&\int_{\rho_-}^{\rho_1} \int_{\tau}^{\rho_1} \frac{p'(r)}{r} dr d\tau \geq \frac{C_1-1}{2} \rho_-\, , \label{eq:integral ineq 2}\\
& \int_{\rho_1}^{\rho_+} \int_{\rho_1}^{\tau} \frac{p'(r)}{r} dr d\tau \geq \frac{C_1-1}{2} \rho_+\, . \label{eq:integral ineq 1}
\end{align}
Now, we introduce two new variables $q_-$ and $q_+$ defined by
$$q_-:= p(\rho_1)-p(\rho_-)\, ,$$
$$q_+:= p(\rho_+)-p(\rho_1)\, .$$
Proving Lemma \ref{l:choice_of_pressure} is then equivalent to showing  the existence of a pressure law $p$ satisfying 
$p(\rho_+)-p(\rho_1)=q_+$, 
$p(\rho_1)-p(\rho_-)=q_-$ and for which the inequalities \eqref{eq:integral ineq 2}-\eqref{eq:integral ineq 1} hold.

First, introducing $f:= p'$, we define the set of functions
$$\mathcal{L}:= \left\{ f\in C^\infty(]0, \infty[, ]0, \infty[)\, :\, 
\int_{\rho_-}^{\rho_1} f= q_- \;\;\mbox{and}\;\; \int_{\rho_1}^{\rho_+} f= q_+\right\}$$
and the two functionals defined on $\mathcal{L}$
$$L^+(f) := \int_{\rho_1}^{\rho_+} \int_{\rho_1}^\tau \frac{f(r)}{r} dr d\tau,$$
$$L^-(f) := \int_{\rho_-}^{\rho_1} \int_\tau^{\rho_1} \frac{f(r)}{r} dr d\tau.$$
Therefore, a sufficient condition for finding a pressure function $p$ with 
the above properties is that 
$$l^+:= \sup_{f\in \mathcal{L}} L^+(f) >\frac{C_1-1}{2} \rho_+$$
and 
$$l^-:= \sup_{f\in \mathcal{L}} L^-(f) >\frac{C_1-1}{2} \rho_- .$$
Let us generalize the space $\mathcal{L}$ as follows.
We introduce 
$$\mathcal{M}^+:= \left\{ \text{positive Radon measures } \mu \text{ on } [\rho_1, \rho_+]: \, \mu([\rho_1, \rho_+])=q_+ \right\},$$
$$\mathcal{M}^-:= \left\{ \text{positive Radon measures } \mu \text{ on } [\rho_-, \rho_1]: \, \mu([\rho_-, \rho_1])=q_- \right\}.$$
For consistency, we extend the functionals $L^+$ and $L^-$ defined on $\mathcal L$ to new functionals $L_+$ and $L_-$ 
respectively defined on $\mathcal M^+$ and on $\mathcal M^-$:
$$L_+(\mu) := \int_{\rho_1}^{\rho_+} \int_{\rho_1}^\tau \frac{1}{r} d\mu(r) d\tau \qquad\qquad \text{ for }\mu\in\mathcal M^+,$$
$$L_-(\mu) := \int_{\rho_-}^{\rho_1} \int_\tau^{\rho_1} \frac{1}{r} d\mu(r) d\tau \qquad\qquad \text{ for }\mu\in\mathcal M^-.$$
Upon setting
$$m^+:= \max_{\mu\in \mathcal M^+} L_+(\mu) $$
and 
$$m^-:= \max_{\mu\in \mathcal M^-} L_-(\mu),$$
it is clear that 
$$l^+\leq m^+ \quad\text{and}\quad l^-\leq m^-.$$
Moreover, let us note that the the maxima $m^pm$ are achieved due to the compactness of $\mathcal M^\pm$
with respect to the weak$^\star$ topology. 
By a simple Fubini type argument, we write 
$$L_+ (\mu)= \int_{\rho_1}^{\rho_+} \frac{\rho_+-r}{r} \; d\mu(r).$$
Hence, defining the function $h\in C([\rho_1, \rho_+])$ as $h(r):=(\rho_+-r)/r$ 
allows us to express the action of the linear functional $L_+$ as
a duality pairing; more precisely we have:
$$L_+(\mu)= <h,\mu> \quad\text{ for }\mu\in \mathcal M^+.$$
Analogously, if we define $g\in C([\rho_-, \rho_1])$ as $g(r):=(r-\rho_-)/r$, we can express $L_-$ as a duality pairing as well:
$$L_-(\mu)= <g,\mu> \quad\text{ for }\mu\in \mathcal M^-.$$
By standard functional analysis, we know that $m^\pm$ must be achieved at the extreme points of $\mathcal M^\pm$. 
The extreme points of $\mathcal M^\pm$ are the single-point measures, i.e. weighted Dirac masses.
For $\mathcal M^+$ the set of extreme points is then given by $E_+:= \left\{ q_+\delta_{\sigma} \text{ for } \sigma \in [\rho_1, \rho_+] \right\}$
while for $\mathcal M^-$ the set of extreme points is then given by $E_-:= \left\{ q_-\delta_{\sigma} \text{ for } \sigma \in [\rho_-, \rho_1] \right\}$
In order to find $m^\pm$, it is sufficient to find the maximum value of $L_\pm$ on $E_\pm$.
Clearly, we obtain
$$m^+= \max_{\sigma\in[\rho_1, \rho_+]} \left\{q_+ \, \frac{\rho_+-\sigma}{\sigma}\right\}=q_+ \, \frac{\rho_+-\rho_1}{\rho_1} $$
and
$$m^-= \max_{\sigma\in[\rho_-, \rho_1]} \left\{q_- \, \frac{\sigma-\rho_-}{\sigma}\right\}=q_- \, \frac{\rho_1-\rho_-}{\rho_1}. $$
Furthermore, given the explicit form of the maximum points, it is rather easy to show that
for every $\varepsilon>0$ there exists a function $f\in \mathcal L$ such that
$$L^+(f) > q_+ \, \frac{\rho_+-\rho_1}{\rho_1} -\varepsilon$$
and 
$$L_-(f) > q_- \, \frac{\rho_1-\rho_-}{\rho_1}-\varepsilon.$$
Such a function $f$ is the derivative of the desired pressure function $p$.
\end{proof}

\subsection{Part II of the proof of Lemma \ref{l:p_strana}} 
We now choose $\rho_1 =1$.
Applying Lemma \ref{l:choice_of_pressure} we set $q_\pm := \pm (p (\rho_\pm) - p (\rho_1))
= \pm (p (\rho_\pm) - p (1))$ and hence
reduce our problem to finding real numbers
$\rho_\pm, \nu_\pm, q_\pm, \alpha, \beta, \gamma, \delta, C_1$ satisfiyng
\begin{equation}\label{eq:surv6_1}
\nu_-<0 <\nu_+\, ,\quad 0 < \rho_-<1<\rho_+\, , \quad q_\pm >0
\end{equation}
\begin{align}
\nu_- (1 - \rho_-) \, &=\, \beta  \label{eq:surv6_2}\\
\nu_- (\rho_- + \alpha) \, &=  \delta \label{eq:surv6_3} \\
\frac{C_1}{2} - \gamma + q_-\, &=\, \nu_- \beta\label{eq:surv6_4}\\
\nu_+ (1-\rho_+ ) \, &=\, \beta \label{eq:surv6_5}\\
\nu_+ (\alpha- \rho_+) \, &= \, \delta \label{eq:surv6_6}\\
\frac{C_1}{2} - \gamma - q_+ &=\, \nu_+  \beta\label{eq:surv6_7}\, ,
\end{align}
\begin{align}
q_- (1-\rho_-)\, &> \frac{C_1-1}{2} \rho_-\label{eq:byebye61}\\
q_+ (\rho_+-1)\, &> \frac{C_1-1}{2} \rho_+\label{eq:byebye62}
\end{align}
and \eqref{eq:sub_trace}-\eqref{eq:sub_det}.

Next, using \eqref{eq:surv6_1}, \eqref{eq:surv6_2} and \eqref{eq:surv6_5} we rewrite
\eqref{eq:byebye61}-\eqref{eq:byebye62} as
\begin{align}
- \beta q_- \, & > \frac{C_1-1}{2} \left(-\nu_- \rho_-\right)\label{eq:byebye71}\\
- \beta q_+ \, & > \frac{C_1-1}{2} \nu_+\rho_+\, .\label{eq:byebye72}
\end{align}
In order to simplify our computations we then introduce the new variables
\begin{equation}\label{eq:new_var}
\overline{\beta} = - \beta\, ,\; \overline{\delta}= - \delta\, , \; \overline{C} = \frac{C_1}{2}\, ,\;
\nu^-= - \nu_-\, , \; r_+ = \rho_+\nu_+ \, \;\mbox{and}\; r_- = \rho_-\nu^- = - \rho_-\nu_-\, .
\end{equation}
Therefore, our conditions become
\begin{equation}\label{eq:surv7_1}
q_\pm, r_\pm, \nu_+, \nu^- > 0  
\end{equation}
\begin{align}
\nu^- - r_- \, &=\, \overline{\beta}  \label{eq:byebye82}\\
r_+ - \nu_+ \, &=\, \overline{\beta} \label{eq:byebye83}\\
r_- + \alpha \nu^- \, &=  \overline{\delta} \label{eq:byebye84} \\
r_+ - \alpha\nu_+ \, &= \, \overline{\delta} \label{eq:byebye85}\\
\overline{C} - \gamma + q_-\, &=\, \nu^- \overline{\beta}\label{eq:surv7_6}\\
\overline{C} - \gamma - q_+ &=\, - \nu_+  \overline{\beta}\label{eq:surv7_7}
\end{align}
\begin{align}
\overline{\beta} q_- &> \left(\overline{C}-\frac{1}{2}\right) r_-\label{eq:surv7_8}\\
\overline{\beta} q_+ &> \left(\overline{C}-\frac{1}{2}\right) r_+\label{eq:surv7_9}\, .
\end{align}
Moreover \eqref{eq:sub_trace}-\eqref{eq:sub_det} become
\begin{align}
&\alpha^2 + \overline{\beta}^2 < 2\overline{C} \label{eq:surv7_10}\\
&(\overline{C} - \alpha^2 +\gamma) ( \overline{C} - \overline{\beta}^2 - \gamma) - (\overline{\delta} -
\alpha \overline{\beta})^2 > 0 \label{eq:surv7_11}\, .
\end{align}
We assume $\alpha^2 \neq 1$ and solve for $\nu_-, \nu^+$ and $r_\pm$ in 
\eqref{eq:byebye82}-\eqref{eq:byebye85} 
to achieve
\begin{align}
&\nu^- = \frac{\overline{\delta} + \overline{\beta}}{1+\alpha}\, , \quad 
\nu_+ = \frac{\overline{\delta} - \overline{\beta}}{1-\alpha}\quad
\mbox{and} \quad r_\pm = \frac{\overline{\delta} - \alpha \overline{\beta}}{1\mp \alpha}\label{eq:solved}\, .
\end{align}
Observe that 
\[
 r_+r_- = \frac{(\overline{\delta}- \alpha \overline{\beta})^2}{1- \alpha^2}\, .
\]
Hence, if we assume $\alpha^2 < 1$ and $\overline{\delta} > \overline{\beta}>0$,
we see that the $\nu^+, \nu_-, r_\pm$ as defined in the formulas \eqref{eq:solved} satisfy
the inequalities in \eqref{eq:surv7_1}. Hence, inserting \eqref{eq:solved} we look
for solutions of the set of identities and inequalities
\begin{align}
& \alpha^2 <1\, , \; \overline{\delta} > \overline{\beta} > 0\, , \; q_\pm >0\label{eq:surv8_1}\\
& \overline{C} - \gamma + q_- = \frac{\overline{\delta}+\overline{\beta}}{1+\alpha} \overline{\beta}\label{eq:byebye8_2}\\
& \overline{C} - \gamma - q_+ = - \frac{\overline{\delta}-\overline{\beta}}{1-\alpha} \overline{\beta}\label{eq:byebye8_3}\\
& \overline{\beta} q_- > \left(\overline{C}-\frac{1}{2}\right)\frac{\overline{\delta} 
- \alpha\overline{\beta}}{1+\alpha}\label{eq:surv8_4}\\
& \overline{\beta} q_+ > \left(\overline{C}-\frac{1}{2}\right)\frac{\overline{\delta} 
- \alpha\overline{\beta}}{1-\alpha}\label{eq:surv8_5}
\end{align}
combined with \eqref{eq:surv7_10} and \eqref{eq:surv7_11}. Observe that, if we assume in
addition that $\bar{C} > \frac{1}{2}$, then $\alpha^2<1$, 
$\overline{\delta}>\overline{\beta}>0$ and \eqref{eq:surv8_4}-\eqref{eq:surv8_5} imply the positivity of
$q_\pm$. We can therefore solve the equations \eqref{eq:byebye8_2}-\eqref{eq:byebye8_3} for $q_\pm$ and
insert the corresponding values in the remaining inequalities \eqref{eq:surv8_4} and \eqref{eq:surv8_5}. Summarizing, we are looking
for $\alpha, \overline{\beta}, \gamma, \overline{\delta}, \overline{C}$ satisfying the following 
inequalities
\begin{align}
& \alpha^2 <1\, ,\; \overline{\delta}>\overline{\beta} > 0\, ,\; \overline{C} > \frac{1}{2}\label{eq:surv9_1}\\
& \overline{\beta} \left[ \overline{\beta} \frac{\overline{\delta}+\overline{\beta}}{1+\alpha} - \overline{C}
+ \gamma \right] > \left(\overline{C}-\frac{1}{2}\right) \frac{\overline{\delta}-
\alpha \overline{\beta}}{1+\alpha}\label{eq:surv9_2}\\
& \overline{\beta} \left[ \overline{\beta} \frac{\overline{\delta}-\overline{\beta}}{1-\alpha} + \overline{C}
- \gamma \right] > \left(\overline{C}-\frac{1}{2}\right) \frac{\overline{\delta}-
\alpha \overline{\beta}}{1-\alpha}\label{eq:surv9_3}\\
&\alpha^2 + \overline{\beta}^2 < 2\overline{C} \label{eq:surv9_4}\\
&(\overline{C} - \alpha^2 +\gamma) ( \overline{C} - \overline{\beta}^2 - \gamma) - (\overline{\delta} -
\alpha \overline{\beta})^2 > 0 \label{eq:surv9_5}\, .
\end{align}
We next introduce the variable $\lambda = \overline{\delta}- \alpha \overline{\beta}$ and rewrite our
inequalities as
\begin{align}
& \alpha^2 <1 \, , \; \lambda > (1-\alpha) \overline{\beta} > 0\, , \; \overline{C} > \frac{1}{2}\label{eq:surv10_1}\\
& \overline{\beta} (1+\alpha) (\overline{\beta}^2 - \overline{C} + \gamma) > \left(\overline{C} - \overline{\beta}^2
- \frac{1}{2}\right)\lambda\label{eq:surv10_2}\\
& \overline{\beta} (1-\alpha) (-\overline{\beta}^2 + \overline{C} - \gamma) > \left(\overline{C} - \overline{\beta}^2
- \frac{1}{2}\right)\lambda\label{eq:surv10_3}\\
& \alpha^2 + \overline{\beta}^2 < 2 \overline{C}\label{eq:surv10_4}\\
&(\overline{C} - \alpha^2 +\gamma) ( \overline{C} - \overline{\beta}^2 - \gamma) > \lambda^2 \label{eq:surv10_5}\, .
\end{align}
Observe that, if we require $\alpha, \overline{\beta}, \gamma$ and $\overline{C}$ to satisfy the following
inequalities
\begin{align}
& \alpha^2 <1\, , \; \overline{C}>\frac{1}{2}\label{eq:surv11_1}\\
&\overline{C} - \alpha^2 + \gamma > 0\label{eq:surv11_2}\\
&\overline{C} - \overline{\beta}^2 - \gamma > 0 \label{eq:surv11_3}\\
&\overline{\beta}^2 + \frac{1}{2} - \overline{C} > 0\label{eq:byebye_immediately}\\
&\sqrt{(\overline{C}-\alpha^2 + \gamma) (\overline{C} - \overline{\beta}^2 - \gamma)} > (1-\alpha) \overline{\beta}
> 0 \label{eq:surv11_4}\\
&\left(\overline{\beta}^2 + \frac{1}{2} - \overline{C}\right) \sqrt{\overline{C}-\alpha^2 + \gamma}
> \overline{\beta} (1+\alpha) \sqrt{\overline{C} - \overline{\beta}^2 - \gamma}\label{eq:surv11_5}
\end{align}
then setting 
\[
\lambda := \sqrt{(\overline{C}-\alpha^2 + \gamma)(\overline{C}-\overline{\beta}^2 - \gamma)} - \eta\, ,
\]
the inequalities \eqref{eq:surv10_1}-\eqref{eq:surv10_5} are satisfied whenever $\eta$ is a sufficiently small
positive number.

Observe next that \eqref{eq:byebye_immediately} is surely satisfied if the remaining inequalities are and
hence we can drop it. Moreover, if $\overline{\beta}, \gamma$ and $\overline{C}$ satisfy
\begin{align}
&\overline{\beta}>0\, , \;\overline{C} > \frac{1}{2}\label{eq:surv12_1}\\
&\overline{C} - \overline{\beta}^2 - \gamma > 0 \label{eq:surv12_2}\\
&\overline{C}-1+\gamma >0\label{eq:surv12_3}\\
&\left(\overline{\beta}^2 + \frac{1}{2} - \overline{C}\right) \sqrt{\overline{C} -1 + \gamma} > 2\overline{\beta}
\sqrt{\overline{C}-\overline{\beta}^2 - \gamma} \label{eq:surv12_4}  
\end{align}
then setting $\alpha = 1 - \vartheta$, the inequalities \eqref{eq:surv11_1}-\eqref{eq:surv11_5} hold provided $\vartheta>0$
is chosen small enough.

Finally, choosing $\overline{C} = \frac{4}{5} \overline{\beta}^2$, $\gamma= - \frac{2}{5} \overline{\beta}^2$
and imposing $\overline{\beta} > \sqrt{\frac{5}{2}}$ we see that \eqref{eq:surv12_1}, \eqref{eq:surv12_2}
and \eqref{eq:surv12_3} are automatically satisfied. Whereas \eqref{eq:surv12_4} is equivalent to
\[
\left(\frac{\overline{\beta}^2}{5} + \frac{1}{2}\right) \sqrt{\frac{2\overline{\beta}^2}{5} -1} 
> \frac{2\overline{\beta}^2}{\sqrt{5}}\, .
\]
However the latter inequality is surely satisfied for $\overline{\beta}$ large enough.

\section{Classical solutions of the Riemann problem}\label{s:riemann}

We show here that, if we restrict our attention to $BV$ selfsimilar solutions of 
\eqref{eq:Euler system}-\eqref{eq:R_data} which do not depend on the variable $x_1$, then for the initial
data of Theorem \ref{t:main} and Corollary \ref{c:compression} the solutions of the Cauchy problem
are unique. We mostly exploit classical results about the $1$-dimensional Riemann problem for hyperbolic system
of conservation laws. We however complement them with some recent results in the theory of transport equations:
the resulting argument is then shorter and moreover yields uniqueness under milder assumptions (see
Remark \ref{r:better} below).

\begin{proposition}\label{p:Riemann_classico}
Consider $p (\rho ) = \rho^2$ and any initial data of type \eqref{eq:R_data} as in Lemma \ref{l:comp_wave}
(which therefore include the data of the proof of Theorem \ref{t:main} and Corollary \ref{c:compression}).
Then there exists a unique admissible self-similar bounded $BV_{loc}$ solution (i.e. of the form
$(\rho, v) (x,t) = (r, w) (\frac{x_2}{t})$) of \eqref{eq:Euler system} with $\rho$ bounded away from $0$. 
\end{proposition}

\begin{remark}\label{r:better}
In fact our proof of Proposition \ref{p:Riemann_classico} has a stronger outcome. In particular
the same uniqueness conclusion holds under the following more general assumptions:
\begin{itemize}
\item $p$ satisfies the usual ``hyperbolicity assumption'' $p' >0$ and the
``genuinely nonlinearity condition'' $2 p' (r) + r p'' (r) >0$ $\forall r>0$;
\item $(\rho, v)$ is a bounded admissible solution with density bounded away from zero, whereas 
the $BV$ regularity and the self-similarity hypotheses are assumed {\em only} for $\rho$ and the second component
of the velocity $v$.
\end{itemize}
\end{remark}

\begin{remark}\label{r:Riemann_classico_2}
The arguments given below can be adapted to show the same uniqueness statement for the Cauchy problem
corresponding to the data generated by Lemma \ref{l:p_strana}. This would only require some lengthier ad hoc analysis of the classical Riemann problem for the system \eqref{eq:2x2}, with $\rho^2$ replaced by the pressures
of Lemma \ref{l:p_strana}. 
\end{remark}

\begin{proof} Observe that the initial data for the first component $v_1$ is the constant $-\frac{1}{\rho_+}$. 
On the other hand:
\begin{itemize}
 \item $\rho$ is a bounded function of locally bounded variation;
\item The vector field $\bar{v} = (0, v_2)$ is bounded, has locally bounded variation and solves the
continuity equation
\begin{equation}\label{eq:cont_100}
\partial_t \rho + \div_x (\rho \bar{v}) = 0\, ; 
\end{equation}
\item $v_1$ is an $L^\infty$ weak solution of the transport equation
\begin{equation}\label{eq:transport}
\left\{
\begin{array}{ll}
&\partial_t (\rho v_1) + \div (\rho \bar{v} v_1) = 0\\ \\
&v_1 (0, \cdot) = -\frac{1}{\rho_+}\, .
\end{array}
\right.
\end{equation}
\end{itemize}
Therefore, the vector field $\bar{v}$ is nearly incompressible in the sense of \cite[Definition 3.6]{dl1}. By the
$BV$ regularity of $\rho$ and $\bar{v}$ we can apply Ambrosio's renormalization theorem \cite[Theorem 4.1]{dl1}
and hence use \cite[Lemma 5.10]{dl1} to infer from \eqref{eq:cont_100} that the pair 
$(\rho, \bar{v})$ has the renormalization property of \cite[Definition 3.9]{dl1}. Thus we can apply 
\cite[Corollary 3.14]{dl1} to infer that there is a unique bounded weak solution of \eqref{eq:transport}.
Since the constant function is a solution, we therefore conclude that $v_1$ is identically equal to 
$-\frac{1}{\rho_+}$. 

Set now $m (x_2, t) := \rho (x_2, t) v_2 (x_2, t)$. The pair $\rho, m$ is then a self-similar $BV_{loc}$ 
weak solution of the $2\times 2$ one-dimensional system of conservation laws
\begin{equation}\label{eq:2x2}
\left\{
\begin{array}{ll}
\partial_t \rho + \partial_{x_2} m = 0 \\ \\
\partial_t m + \partial_{x_2} \left(\frac{m^2}{\rho} + \rho^2\right) = 0\, ,        
\end{array}
\right. 
\end{equation}
that is the standard system of isentropic Euler in Eulerian coordinates with a particular 
polytropic pressure. It is well-known that such system is genuinely nonlinear
in the sense of \cite[Definition 7.5.1]{da} and therefore, following the discusssion of \cite[Section 9.1]{da}
we conclude that the functions $(\rho, m)$ result from ``patching'' rarefaction waves and shocks 
connecting constant
states, i.e. they are classical solutions of the so-called Riemann problem in the sense of \cite[Section 9.3]{da}. 
It is well-known
that in the special case of \eqref{eq:2x2} the latter property and the admissibility condition
determine uniquely the functions $(\rho, m)$.
For instance, one can apply \cite[Theorem 3.2]{kk}.
\end{proof}

\end{document}

%% file: fig11.pdf_t
\begin{picture}(0,0)%
\includegraphics{fig11.pdf}%
\end{picture}%
\setlength{\unitlength}{1973sp}%
\begingroup\makeatletter\ifx\SetFigFont\undefined%
\gdef\SetFigFont#1#2#3#4#5{%
  \reset@font\fontsize{#1}{#2pt}%
  \fontfamily{#3}\fontseries{#4}\fontshape{#5}%
  \selectfont}%
\fi\endgroup%
\begin{picture}(11241,6966)(718,-7273)
\put(10951,-6811){\makebox(0,0)[lb]{\smash{{\SetFigFont{12}{14.4}{\rmdefault}{\mddefault}{\updefault}{\color[rgb]{0,0,0}$x_2$}%
}}}}
\put(5551,-586){\makebox(0,0)[lb]{\smash{{\SetFigFont{12}{14.4}{\rmdefault}{\mddefault}{\updefault}{\color[rgb]{0,0,0}$t$}%
}}}}
\put(1276,-5236){\makebox(0,0)[lb]{\smash{{\SetFigFont{17}{20.4}{\rmdefault}{\mddefault}{\updefault}{\color[rgb]{0,0,0}$P_-$}%
}}}}
\put(10126,-5611){\makebox(0,0)[lb]{\smash{{\SetFigFont{17}{20.4}{\rmdefault}{\mddefault}{\updefault}{\color[rgb]{0,0,0}$P_+$}%
}}}}
\put(2776,-2986){\makebox(0,0)[lb]{\smash{{\SetFigFont{17}{20.4}{\rmdefault}{\mddefault}{\updefault}{\color[rgb]{0,0,0}$P_1$}%
}}}}
\put(7426,-2161){\makebox(0,0)[lb]{\smash{{\SetFigFont{17}{20.4}{\rmdefault}{\mddefault}{\updefault}{\color[rgb]{0,0,0}$P_2$}%
}}}}
\put(9226,-3211){\makebox(0,0)[lb]{\smash{{\SetFigFont{17}{20.4}{\rmdefault}{\mddefault}{\updefault}{\color[rgb]{0,0,0}$P_3$}%
}}}}
\end{picture}%

%% file: fig10.pdf_t
\begin{picture}(0,0)%
\includegraphics{fig10.pdf}%
\end{picture}%
\setlength{\unitlength}{1973sp}%
\begingroup\makeatletter\ifx\SetFigFont\undefined%
\gdef\SetFigFont#1#2#3#4#5{%
  \reset@font\fontsize{#1}{#2pt}%
  \fontfamily{#3}\fontseries{#4}\fontshape{#5}%
  \selectfont}%
\fi\endgroup%
\begin{picture}(10224,6966)(1189,-7273)
\put(10951,-6811){\makebox(0,0)[lb]{\smash{{\SetFigFont{12}{14.4}{\rmdefault}{\mddefault}{\updefault}{\color[rgb]{0,0,0}$x_2$}%
}}}}
\put(7576,-6886){\makebox(0,0)[lb]{\smash{{\SetFigFont{12}{14.4}{\rmdefault}{\mddefault}{\updefault}{\color[rgb]{0,0,0}$\nu_+$}%
}}}}
\put(5326,-6886){\makebox(0,0)[lb]{\smash{{\SetFigFont{12}{14.4}{\rmdefault}{\mddefault}{\updefault}{\color[rgb]{0,0,0}$\nu_-$}%
}}}}
\put(6076,-4411){\makebox(0,0)[lb]{\smash{{\SetFigFont{12}{14.4}{\rmdefault}{\mddefault}{\updefault}{\color[rgb]{0,0,0}$1$}%
}}}}
\put(6976,-2461){\makebox(0,0)[lb]{\smash{{\SetFigFont{17}{20.4}{\rmdefault}{\mddefault}{\updefault}{\color[rgb]{0,0,0}$P_1$}%
}}}}
\put(9451,-5011){\makebox(0,0)[lb]{\smash{{\SetFigFont{17}{20.4}{\rmdefault}{\mddefault}{\updefault}{\color[rgb]{0,0,0}$P_+$}%
}}}}
\put(2251,-4111){\makebox(0,0)[lb]{\smash{{\SetFigFont{17}{20.4}{\rmdefault}{\mddefault}{\updefault}{\color[rgb]{0,0,0}$P_-$}%
}}}}
\put(5551,-586){\makebox(0,0)[lb]{\smash{{\SetFigFont{12}{14.4}{\rmdefault}{\mddefault}{\updefault}{\color[rgb]{0,0,0}$t$}%
}}}}
\end{picture}%